\documentclass[12pt,twoside]{amsart}
\usepackage{amsthm,amsfonts,amssymb,amscd,euscript,epsfig}
\usepackage{tikz}
\input{xy}
\xyoption{all}

\newcommand{\msc}{\mathrm{sc}_m}

\DeclareMathOperator{\Par}{Par}

\newcommand{\Q}{{\mathbb Q}}

\newcommand{\N}{{\mathbb N}}
\newcommand{\R}{{\mathbb R}}
\newcommand{\W}{W}

\newlength{\cellsize}
\newcommand\tableau[1]{
\vcenter{
\let\\=\cr
\baselineskip=-16000pt
\lineskiplimit=16000pt
\lineskip=0pt
\halign{&\tableaucell{##}\cr#1\crcr}}}

\newcommand{\tableaucell}[1]{{%
\def \arg{#1}\def \void{}%
\ifx \void \arg
\vbox to \cellsize{\vfil \hrule width \cellsize height 0pt}%
\else
\unitlength=\cellsize
\begin{picture}(1,1)
\put(0,0){\makebox(1,1){$#1$}}
\put(0,0){\line(1,0){1}}
\put(0,1){\line(1,0){1}}
\put(0,0){\line(0,1){1}}
\put(1,0){\line(0,1){1}}
\end{picture}%
\fi}}

\theoremstyle{plain}
\newtheorem{thm}{Theorem}%[section]

\newtheorem{lem}[thm]{Lemma}

\newtheorem{prop}[thm]{Proposition}

\newtheorem{conj}[thm]{Conjecture}

\theoremstyle{definition}
\newtheorem{defn}[thm]{Definition}

\newtheorem{example}[thm]{Example}

\newtheorem{rem}[thm]{Remark}

\newtheorem{ack}{Acknowledgments}  
\theoremstyle{remark}

\def\area{{\rm area}}
\def\dinv{{\rm dinv}}
\def\qdeg{q\textrm{-}{\rm deg}}
\def\tdeg{t\textrm{-}{\rm deg}}

\begin{document}

\title[Combinatorics of higher $q,t$-Catalan polynomials]
{Combinatorics of certain higher $q,t$-Catalan polynomials:
 chains, joint symmetry, and the Garsia-Haiman formula} 
\author{Kyungyong Lee, Li Li, and Nicholas A. Loehr}
\thanks{Research of K.L. is partially supported by NSF grant DMS 0901367.}
\thanks{This work was partially supported by a grant from the Simons
Foundation (\#244398 to Nicholas Loehr).}

\address{Department of Mathematics, Wayne State University, Detroit, MI 48202}
\email{{\tt klee@math.wayne.edu}}
\address{Department of Mathematics and Statistics, Oakland University, Rochester, MI 48309}
\email{{\tt li2345@oakland.edu}}
\address{Department of Mathematics, Virginia Tech, Blacksburg, VA 24061;
and Mathematics Department, United States Naval Academy, Annapolis, MD 21402}
\email{{\tt nloehr@math.vt.edu}, {\tt loehr@usna.edu}}

\maketitle

\begin{abstract}
The higher $q,t$-Catalan polynomial $C^{(m)}_n(q,t)$
can be defined combinatorially as a weighted sum 
of lattice paths contained in certain triangles,
or algebraically as a complicated sum of rational functions
indexed by partitions of $n$. This paper proves the equivalence
of the two definitions for all $m\geq 1$ and all $n\leq 4$.
We also give a bijective proof of the joint symmetry property 
$C^{(m)}_n(q,t)=C^{(m)}_n(t,q)$ for all $m\geq 1$ and all $n\leq 4$. 
The proof is based on a general approach for proving joint symmetry
that dissects a collection of objects into chains, and then passes from 
a joint symmetry property of initial points and terminal points
to joint symmetry of the full set of objects.  
Further consequences include unimodality results and specific formulas 
for the coefficients in $C^{(m)}_n(q,t)$ for all $m\geq 1$ and all $n\leq 4$.
We give analogous results for certain rational-slope $q,t$-Catalan polynomials. 
\end{abstract}

\section{Introduction}
\label{sec:intro}

\subsection{The $q,t$-Catalan Polynomials}
\label{subsec:intro-qtcat}

The \emph{$q,t$-Catalan polynomials} $C_n(q,t)$,
introduced by Garsia and Haiman~\cite{GH-qtcat} in 1996,
play a prominent role in combinatorics, symmetric function theory,
and algebraic geometry. These polynomials can be defined combinatorially
as follows. A sequence $\gamma=(\gamma_0,\gamma_1,\ldots,\gamma_{n-1})$
is called a \emph{Dyck word} if and only if $\gamma_0=0$, each
$\gamma_i\in\N=\{0,1,2,\ldots\}$, and $\gamma_i\leq \gamma_{i-1}+1$
for $1\leq i<n$. Let $\W_n$ be the set of Dyck words of length $n$.
For $\gamma\in \W_n$, define $\area(\gamma)=\sum_{i=0}^{n-1} \gamma_i$,
and define $\dinv(\gamma)$ to be the number of $(i,j)$ with
$0\leq i<j<n$ and $\gamma_i-\gamma_j\in\{0,1\}$. Then
\[ C_n(q,t)=\sum_{\gamma\in \W_n} q^{\area(\gamma)}t^{\dinv(\gamma)}. \]
For example, when $n=3$, $\W_3=\{(0,0,0), (0,0,1), (0,1,0), (0,1,1),
(0,1,2)\}$, and $C_3(q,t)=t^3+qt+qt^2+q^2t+q^3$.

Part of the interest of the $q,t$-Catalan polynomials is that there
are many different ways of defining $C_n(q,t)$; the equivalence
of these definitions is a deep result of algebraic combinatorics
due to Garsia, Haglund, and Haiman~\cite{haglund-book}. We have defined 
$C_n(q,t)$ as a weighted sum of Dyck words. Another combinatorial
formula, first proposed by Haglund~\cite{hagqtcatconj}, expresses
$C_n(q,t)$ as a sum of Dyck paths weighted by area and Haglund's
bounce statistic. Garsia and Haiman's original definition~\cite{GH-qtcat} 
presented $C_n(q,t)$ as a complicated sum of rational functions in $q$ and 
$t$ indexed by integer partitions of $n$ (see below for more details). 
The Garsia-Haiman formula
identifies $C_n(q,t)$ as the coefficient of the sign character
in $\nabla(e_n)$, where $e_n$ is an elementary symmetric polynomial
and $\nabla$ is the nabla operator of Bergeron and Garsia~\cite{BG-scifi}.
In turn, Haiman proved that $\nabla(e_n)$ is the Frobenius series
of the diagonal harmonics module in $2n$ variables~\cite{haiman-vanish}, 
so that $C_n(q,t)$ can be defined algebraically as the Hilbert series of the 
module of diagonal harmonic alternants. There are other geometric
manifestations of $C_n(q,t)$ involving Hilbert schemes~\cite{haiman-tqcat}
and, more recently, compactified Jacobians of plane curve 
singularities~\cite{GM1,GM2}.

\subsection{The Higher $q,t$-Catalan Polynomials}
\label{subsec:intro-higher-qtcat}

The \emph{higher $q,t$-Catalan polynomials} are generalizations
of the $q,t$-Catalan polynomials that depend on two integer parameters
$m$ and $n$; they reduce to ordinary $q,t$-Catalan polynomials when $m=1$.
We first review the combinatorial definition of the higher $q,t$-Catalan
polynomials given in~\cite{L-mcat}, which generalizes the formula
for $C_n(q,t)$ as a weighted sum of Dyck words. 
Fix an integer $m\geq 1$.  An \emph{$m$-Dyck word} is a sequence
$\gamma=(\gamma_0,\gamma_1, \ldots,\gamma_{n-1})$ such that $\gamma_i\in\N$,
$\gamma_0=0$, and $\gamma_i\leq \gamma_{i-1}+m$ for $1\leq i<n$.  
Denote by $W_n^{(m)}$ the set of $m$-Dyck words of length $n$.
For $\gamma\in W_n^{(m)}$, define $\area(\gamma)=\sum_{i=0}^{n-1} \gamma_i$,
and define $\dinv_m(\gamma)=\sum_{0\le i<j<n} \msc(\gamma_i-\gamma_j)$,
where 
$$ \msc(p)=
\left\{
 \begin{array}{ll}
   m+1-p, &\hbox{if $1\le p\le m$};\\
   m+p,  &\hbox{if $-m\le p\le 0$};\\
   0, &\hbox{for all other $p$}.
 \end{array}
 \right.
$$ 
Define $C^{(m)}_n(q,t)
=\sum_{\gamma\in\W_n^{(m)}} q^{\area(\gamma)}t^{\dinv_m(\gamma)}$. 
See~\cite{L-mcat} for an equivalent combinatorial definition
of $C^{(m)}_n(q,t)$ using $m$-Dyck paths weighted by area and
a suitable $m$-bounce statistic.

One can also give algebraic definitions of the higher
$q,t$-Catalan polynomials. However, for $m>1$, the algebraic
definitions are not yet known to be equivalent to the combinatorial
definitions. Thus we will use the notation $AC^{(m)}_n(q,t)$ to
denote the algebraic version of the higher $q,t$-Catalan polynomials.
These can be defined in terms of the nabla operator by setting
\[ AC^{(m)}_n(q,t)=\langle \nabla^m(e_n),s_{(1^n)}\rangle. \]
Garsia and Haiman~\cite{GH-qtcat} gave an explicit
formula for $AC^{(m)}_n(q,t)$ in their original paper on $q,t$-Catalan
polynomials, which we now describe.

Recall that a \emph{partition of $n$} is a sequence
$\mu=(\mu_1,\mu_2,\ldots,\mu_s)$ of positive integers 
with $\mu_1\geq\mu_2\geq\cdots\geq\mu_s$ and $\mu_1+\cdots+\mu_s=n$.  
Let $\Par(n)$ be the set of partitions of $n$.
For $\mu\in\Par(n)$, the \emph{diagram of $\mu$} is the set
\[ D(\mu)=\{(i,j)\in\N\times\N: 1\leq i\leq s,1\leq j\leq\mu_i\}, \]
which can be visualized as a set of left-justified cells
with $\mu_i$ squares in the $i$'th row from the top.
For example, $\mu=(4,4,3,1)$ is a partition in $\Par(12)$ with
\[ D(\mu)= \tableau{ 
{}&{}&{}&{}\\
{*}&{}&{}&{}\\
{}&{}&{}\\
{} }. \]
Fix a cell $c\in D(\mu)$.
The \emph{arm of $c$}, denoted $a(c)$, 
is the number of cells in $D(\mu)$ to the right of $c$ in the same row.
The \emph{coarm of $c$}, denoted $a'(c)$, 
is the number of cells in $D(\mu)$ to the left of $c$ in the same row.
The \emph{leg of $c$}, denoted $l(c)$, 
is the number of cells in $D(\mu)$ below $c$ in the same column.
The \emph{coleg of $c$}, denoted $l'(c)$, 
is the number of cells in $D(\mu)$ above $c$ in the same column.
In the example shown above, the cell $c=(2,1)$ (marked by an asterisk
in the figure) has $a(c)=3$, $a'(c)=0$, $l(c)=2$, and $l'(c)=1$.

The Garsia-Haiman formula for $AC^{(m)}_n(q,t)$ is a sum of rational
functions in $q$ and $t$ indexed by partitions of $n$, which is assembled
from the following ingredients. For each $\mu\in\Par(n)$, define:
\[ T_{\mu} = \prod_c [q^{a'(c)}t^{l'(c)}]; \qquad
 B_{\mu} = \sum_c q^{a'(c)}t^{l'(c)}; \qquad
\Pi_{\mu}= \prod_{c\neq (1,1)} (1-q^{a'(c)}t^{l'(c)}); \]
\[ w_{\mu} = \prod_c [(q^{a(c)}-t^{l(c)+1})(t^{l(c)}-q^{a(c)+1})]. \] 
All sums and products here range over cells $c$ in $D(\mu)$, except
the summation for $\Pi_{\mu}$ excludes the upper-left corner cell $c=(1,1)$.
Garsia and Haiman's original definition of the higher $q,t$-Catalan 
polynomials is:
\begin{equation}\label{eq:GH-formula}
AC^{(m)}_n(q,t)=\sum_{\mu\in\Par(n)} 
\frac{T_{\mu}^{m+1}(1-q)(1-t)B_{\mu}\Pi_{\mu}}{w_{\mu}}.
\end{equation}

The following conjecture has been open since approximately 2001.
(The $m=1$ case follows from the difficult theorem of Garsia, Haglund,
and Haiman mentioned above.)

\begin{conj}[Haiman/Loehr~\cite{L-mcat}]\label{conj:qtcat}
For all $m,n\in\N^+$, $$C^{(m)}_n(q,t)=AC^{(m)}_n(q,t).$$
\end{conj}

The first main goal of this paper
is to prove this conjecture for all $m\geq 1$ 
and all $n\leq 4$. Our proof is rather intricate, but it
requires only elementary combinatorial operations on $m$-Dyck words 
and algebraic manipulations of expressions involving $q$ and $t$.
The proof will evolve as a consequence of our combinatorial 
investigation of joint symmetry, which we describe next.

\subsection{Joint Symmetry}\label{subsec:intro-joint-symm}
One notable feature of the higher $q,t$-Catalan polynomials is 
the \emph{joint symmetry} $C_n^{(m)}(q,t)=C_n^{(m)}(t,q)$.  
The joint symmetry property for $AC_n^{(m)}(q,t)$
follows fairly easily from the Garsia-Haiman definition.
For, letting $\mu'$ denote the conjugate of the partition $\mu$,
it is immediate from the definitions that
$T_{\mu'}(q,t)=T_{\mu}(t,q)$, $B_{\mu'}(q,t)=B_{\mu}(t,q)$,
$\Pi_{\mu}(q,t)=\Pi_{\mu'}(t,q)$, and $w_{\mu'}(q,t)=w_{\mu}(t,q)$.
Joint symmetry of $AC_n^{(m)}$ now follows by replacing the
summation variable $\mu$ by $\mu'$ in~\eqref{eq:GH-formula}
and using the preceding identities.

On the other hand, the joint symmetry property for the 
combinatorially defined polynomials $C_n^{(m)}(q,t)$ is a great mystery.
It is an open problem to define a collection of involutions
$f:\W_n^{(m)}\rightarrow \W_n^{(m)}$ for all $m,n\in\N^+$, such that
$\area(f(\gamma))=\dinv_m(\gamma)$ and 
$\dinv_m(f(\gamma))=\area(\gamma)$ for all $\gamma\in \W_n^{(m)}$.
This problem is open even for $m=1$.
For all $m,n\in\N^+$, a bijection is known~\cite{haglund-book,L-mcat} 
proving the weaker \emph{univariate symmetry property} 
$C_n^{(m)}(q,1)=C_n^{(m)}(1,q)$.

The second main goal of this paper is to develop a combinatorial
framework for understanding the joint symmetry of $C^{(m)}_n(q,t)$
and the \emph{rational-slope $q,t$-Catalan polynomials} $C_{r,s,n}(q,t)$
(to be defined later).
We introduce a strategy for producing bijective proofs of joint symmetry
that involves dissecting a set of objects into a disjoint union of chains.
We show that if the set of initial points and the set of terminal points 
of all the chains are linked by a certain joint symmetry property,
then the generating function for the entire set is jointly symmetric.
We use our method to give a combinatorial proof of the joint
symmetry of $C^{(m)}_n(q,t)$ and $C_{r,s,n}(q,t)$ 
for triangles of height at most 4 
and all choices of the slope parameters $m$, $r$, and $s$. 
Gorsky and Mazin~\cite{GM2} recently gave a combinatorial
proof of joint symmetry for triangles of height 3 
using a different approach.
Our method has the added benefit of providing explicit formulas for all 
the coefficients of $C^{(m)}_n(q,t)$ for $n\leq 4$, as well as
providing the foundation for our proof of Conjecture~\ref{conj:qtcat}
for these values of $n$.  We also obtain some unimodality results
for the coefficient sequences obtained by looking at monomials
in the higher $q,t$-Catalan polynomials of a given total degree.
Many of the ingredients in our proof extend to triangles of arbitrary sizes,
although we cannot yet prove joint symmetry in full generality.

The rest of the paper is organized as follows.  
Section~\ref{sec:symm-chains} describes our general approach
of building chains and then passing from the $q,t$-symmetry of chain 
endpoints to $q,t$-symmetry of the full set of combinatorial objects.  
Section~\ref{sec:chain-mcat} defines a map 
$f_0$ that will be used to construct chains of $m$-Dyck words.  
Section~\ref{sec:proof-jsymm} proves joint symmetry 
of $C^{(m)}_n(q,t)$ for all $n\leq 4$.
Section~\ref{sec:proof-GH} proves Conjecture~\ref{conj:qtcat}
for all $n\leq 4$.  
Section~\ref{sec:rat-slope} gives the definition of 
rational slope $q,t$-Catalan polynomials and extends 
the previous constructions to this situation. We also
compare our method to the Gorsky-Mazin proof for triangles of height 3.
Finally, Section~\ref{sec:n=5} gives a conjectured chain map
for triangles of height 5 and further discussion of the challenges that arise
for larger $n$.

\section{Joint Symmetry via Chains}
\label{sec:symm-chains}

This section considers the following general situation.
We are given a finite set $W$ and two statistics
$a:W\rightarrow\N$ and $d:W\rightarrow\N$. For any
subset $S$ of $W$, define
\[ C_S(q,t)=\sum_{w\in S} q^{a(w)}t^{d(w)}. \]
We are also given a set $I$ of \emph{initial objects} in $W$,
a set $T$ of \emph{terminal objects} in $W$, and a bijection
$f:W\setminus T\rightarrow W\setminus I$ such that 
\[ a(f(w))=a(w)-1\mbox{ and }d(f(w))=d(w)+1
  \mbox{ for all $w\in W\setminus T$.} \]
Then $W$ is the disjoint union of \emph{$f$-chains},
each of which proceeds from an object in $I$ to an object in $T$.
More specifically, for each $w_0\in I$, there is an $f$-chain
\[ w_0 {\longrightarrow} w_1
{\longrightarrow} w_2
{\longrightarrow} \cdots
{\longrightarrow} w_k, \]
where $w_{i+1}=f(w_i)$ for $0\leq i<k$, and
$k\in\N$ is minimal such that $w_k\in T$.
(If $w_0\in I\cap T$, then $k=0$ and the $f$-chain consists of $w_0$ alone.)
%Write $\chain{w_0}=\{w_0,w_1,\ldots,w_k\}=\chain{w_k}$ for the $f$-chain from $w_0\in I$ to $w_k\in T$.

\begin{thm}\label{thm:chain-symm}
With the above notation, if $C_T(q,t)=C_I(t,q)$, then $C_W(q,t)=C_W(t,q)$.
\end{thm}

The next two subsections give an algebraic proof and a bijective proof
of this theorem. 

\subsection{Algebraic Proof of Theorem~\ref{thm:chain-symm}.}
\label{subsec:alg-chain-symm}
The following lemma and its proof were communicated to us
 by Mikhail Mazin. Theorem \ref{thm:chain-symm} immediately follows from this lemma.
\begin{lem}\label{lem:CI-to-CW}
{\rm (a)} For any finite set $W$ with a chain map $f$, we have
$$
C_W(q,t)=\frac{C_I(q,t)}{1-t/q}
+\frac{C_T(q,t)}{1-q/t}.
$$
In particular, if $C_I(q,t)=C_T(q,t)$ then $C_W(q,t)=C_I(q,t)$.

{\rm (b)} Assume $C_T(q,t)=C_I(t,q)$. 
Then
\[ C_W(q,t)=\frac{C_I(q,t)}{1-t/q}+\frac{C_I(t,q)}{1-q/t}
  =\sigma\left(\frac{C_I(q,t)}{1-t/q}\right), \]
where  $\sigma:\Q(q,t)\rightarrow\Q(q,t)$ is the map given by
$\sigma(F(q,t))=F(q,t)+F(t,q)$.    
\end{lem}
\begin{proof}
{\rm (a)} For each $f$-chain $C_i$:
$w_0^i {\longrightarrow} w_1^i
{\longrightarrow} w_2^i
{\longrightarrow} \cdots
{\longrightarrow} w_{k_i}^i$,
apply the formula for geometric progressions to obtain
$$
C_{C_i}(q,t)=\frac{C_{\{w_0^i\}}(q,t)}{1-t/q}
-\frac{C_{\{w_{k_i}^i\}}(q,t)\cdot t/q}{1-t/q}
=\frac{C_{\{w_0^i\}}(q,t)}{1-t/q}
+\frac{C_{\{w_{k_i}^i\}}(q,t)}{1-q/t}.
$$
After summing up over all chains, one gets
$$
C_W(q,t)=\frac{C_I(q,t)}{1-t/q}
+\frac{C_T(q,t)}{1-q/t}.
$$

{\rm (b)} is immediate from (a).
\end{proof}

\begin{rem}\label{rem:coefficient}
Define $|C_I(q,t)|^{\deg=j+k}_{\qdeg\ge j}$ to be the number of 
terms in the sum defining $C_I(q,t)$ whose total degree is $j+k$ and 
whose $q$-degree is at least $j$, and define  $|C_I(q,t)|^{\deg=j+k}_{\tdeg> j}$ similarly. From Lemma \ref{lem:CI-to-CW} (a), 
$$
C_W(q,t)=\frac{C_I(q,t)}{1-t/q}
-\frac{C_T(q,t)\cdot t/q}{1-t/q}.
$$
Using power series expansion we conclude that the coefficient of $q^jt^k$ in $C_W(q,t)$ is  
$$\Big|C_I(q,t)\Big|^{\deg=j+k}_{\qdeg\ge j}-\Big|C_T(q,t)\Big|^{\deg=j+k}_{\qdeg> j}.$$
(A more explanatory proof: the coefficient is equal to the number of chains
with the same total degree starting ``weakly before" the monomial
minus the number of such chains ending ``strictly before" the monomial.)
In particular if $C_T(q,t)=C_I(t,q)$, then the coefficient is
$$\Big|C_I(q,t)\Big|^{\deg=j+k}_{\qdeg\ge j}-\Big|C_I(q,t)\Big|^{\deg=j+k}_{\tdeg> j}.$$
\end{rem}

\subsection{Bijective Proof of Theorem~\ref{thm:chain-symm}.}
\label{subsec:bij-chain-symm}

We now show that any bijective proof of the hypothesis
 $C_T(q,t)=C_I(t,q)$ of Theorem~\ref{thm:chain-symm}
can be converted to a bijective proof of the conclusion
 $C_W(q,t)=C_W(t,q)$. More specifically,
assume that we are given a bijection $h:T\rightarrow I$ 
such that $a(h(w))=d(w)$ and $d(h(w))=a(w)$ for all $w\in T$.  
We will use $h$ and $f$ to build a canonical
involution $J:W\rightarrow W$ such that $a(J(w))=d(w)$ 
and $d(J(w))=a(w)$ for all $w\in W$. 

First, since $f:W\setminus T\rightarrow W\setminus I$ and $h:T\rightarrow I$
are bijections, $f\cup h$ is a bijection from $W$ to $W$ 
(here we are viewing functions as sets of ordered pairs).
The \emph{digraph} of $f\cup h$ is the directed graph $G$ with
vertex set $W$ and directed edges $w\rightarrow f(w)$ for all $w\in W\setminus T$
and $w\rightarrow h(w)$ for all $w\in T$. Because $f\cup h$ is a bijection,
$G$ is a disjoint union of directed cycles. 

To build the involution $J$, we perform the following construction
on each directed cycle $C$ in $G$ (see the examples below for illustrations).
If all vertices $w$ on $C$ satisfy $a(w)=d(w)$ --- which must occur
when $C$ consists of just one vertex --- define $J(w)=w$ for all such $w$.
For all other $C$, we will create a drawing of the cycle $C$
in the first quadrant of the $xy$-plane, as follows. First observe
that the value of $a(w)+d(w)$ for all vertices $w$ on $C$ is constant,
because of the properties of $f$ and $h$. In our drawing of $C$,
each $w\in C$ will be drawn at a lattice point $(x(w),y(w))$ such that:
$y(w)=|a(w)-d(w)|$; the lattice point for $w$ is colored black
if $a(w)-d(w)>0$; and the lattice point for $w$ is colored white
if $a(w)-d(w)<0$. (When $y(w)=0$, the lattice point for $w$ may be black 
or white.)

To create the drawing, pick any $w_0\in I\cap C$ that maximizes $a(w)-d(w)$,
and draw a black dot for $w_0$ at $(0,a(w_0)-d(w_0))$. 
One may routinely check that $a(w_0)-d(w_0)>0$.  Let the 
distinct vertices on the cycle $C$ (in order) be $w_0,w_1,\ldots,w_k$. Assume 
by induction, for some fixed $i<k$, that we have already drawn dots for
$w_0,w_1,\ldots,w_i$ at respective coordinates $(0,y_0)$, $(1,y_1)$, $\ldots$, 
$(i,y_i)$. To continue the drawing, consider various cases.
\begin{itemize}
\item Case~1: $w_i\not\in T$, so $w_{i+1}=f(w_i)$.
\begin{itemize}
\item Case~1a: $(i,y_i)$ is a black dot and $y_i>1$.
 Draw a black dot for $w_{i+1}$ at $(i+1,y_i-2)$.
\item Case~1b: $(i,y_i)$ is a black dot and $y_i=1$.
 Draw a white dot for $w_{i+1}$ at $(i+1,1)$.
 % [Intuition: Color changed because we reflected off the $x$-axis.]
\item Case~1c: $(i,y_i)$ satisfies $y_i=0$.
 Draw a white dot for $w_{i+1}$ at $(i+1,2)$.
 % [Intuition: Color changed because we reflected off the $x$-axis.]
\item Case~1d: $(i,y_i)$ is a white dot.
 Draw a white dot for $w_{i+1}$ at $(i+1,y_i+2)$.
\end{itemize}
\item Case~2: $w_i\in T$, so $w_{i+1}=h(w_i)\in I$.
 Draw a dot for $w_{i+1}$ of the opposite color as the dot for $w_i$
 at $(i+1,y_i)$.  
 % [Intuition: color change implements the $q,t$-symmetry property of $h$.]
\end{itemize}
In the examples pictured below, we also draw line segments between successive
dots to help visualize the cycle. One sees that dots are black when
the line segments are moving down, and dots are white when the 
line segments are moving up. Color changes occur at horizontal line
segments and also due to reflection off the bottom boundary $y=0$.

One may now check (by induction on $i$) that the properties
stated earlier regarding $y(w_i)$ and the color of the dot for $w_i$
are indeed true. One should further check that the dot for $w_k$
is a white dot at $y$-coordinate $y_k=y_0$ (since $h(w_k)=w_0$) 
and that no intermediate $y$-coordinate
$y_i$ exceeds $y_0$.  Finally, for each black dot $(i,y_i)$ in 
the drawing with $y_i>0$, there is a unique leftmost white dot $(j,y_i)$ at 
the same level (with $j>i$) that can be found pictorially by moving due east 
from $(i,y_i)$ until hitting the next dot.  
Define $J(w_i)=w_j$ and $J(w_j)=w_i$ for each such ``matched pair'' of a 
black dot and a later white dot.  
Define $J(w)=w$ for any $w$ with $y(w)=0$ 
(which holds if and only if $a(w)=d(w)$).
The properties of heights and colors of dots show that $J$ interchanges
$a$ and $d$.  Moreover, $J$ does not depend on the initial choice of $w_0$, 
since choosing a different $w_0'\in I\cap C$ maximizing $a(w)-d(w)$ will 
produce a cyclically shifted version of the original drawing with the same
matchings of black dots to white dots.  In this sense, $J$ is canonically
determined from the given maps $f$ and $h$.

\begin{example}
Let $W=\{1,2,\ldots,15\}$, and define $a,d,f,h$ as shown in this table:
%\[ \begin{array}{|c|c|c|c|c|c|c|c|c|c|c|c|c|c|c|c|}
\[  \begin{array}{|c|r|r|r|r|r|r|r|r|r|r|r|r|r|r|r|}
  \hline    w & 1 & 2 & 3 & 4 & 5 & 6 & 7 & 8 & 9 & 10 & 11 & 12 & 13 & 14 & 15
\\\hline
         a(w) & 7 & 6 & 5 & 6 & 5 & 4 & 3 & 2 & 1 & 3  &  2 &  7 &  6 &  2 &  1
\\ d(w) & 1 & 2 & 3 & 2 & 3 & 4 & 5 & 6 & 7 & 5  &  6 &  1 &  2 &  6 &  7
\\ a(w)-d(w) 
              & 6 & 4 & 2 & 4 & 2 & 0 &-2 &-4 &-6 &-2  & -4 &  6 &  4 & -4 & -6
\\ f(w) & 2 & 3 & - & 5 & 6 & 7 & 8 & 9 & - & 11 &  - & 13 &  - & 15 & -
\\ h(w) & - & - & 10& - & - & - & - & - & 1 &  - &  4 &  - & 14 & -  & 12
\\\hline
\end{array} \]
Note $I=\{1,4,10,12,14\}$ and $T=\{3,9,11,13,15\}$.
The left side of Figure~\ref{fig:chbij1} shows the $f$-chains drawn
vertically, with each $w\in W$ drawn at height $a(w)-d(w)$.
On the right of the figure, we draw the two cycles of $f\cup h$
in the first quadrant as described above. The horizontal arrows
indicate the action of $J$, namely:
\[ J:\ \ 
   1\leftrightarrow 9,\ \ 
   2\leftrightarrow 11,\ \ 
   3\leftrightarrow 10,\ \ 
   4\leftrightarrow 8,\ \ 
   5\leftrightarrow 7,\ \ 
   6\leftrightarrow 6,\ \ 
   12\leftrightarrow 15,\ \ 
   13\leftrightarrow 14. \]
\begin{center}
\begin{figure}
\epsfig{file=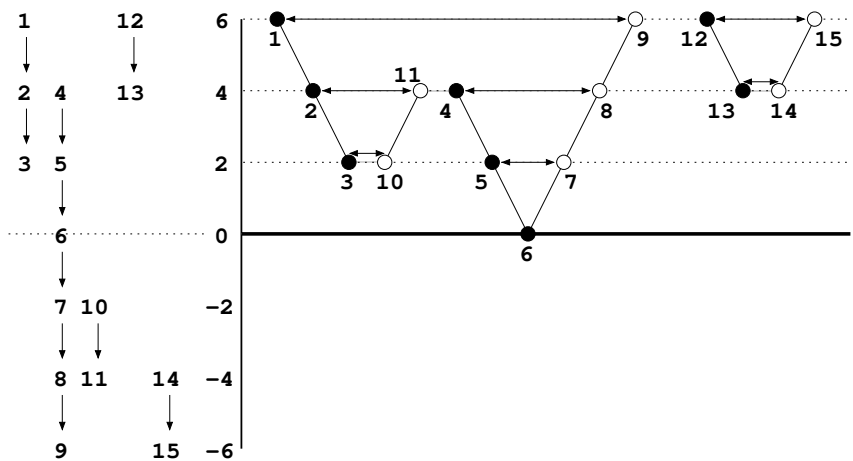,scale=0.75}
\caption{Example of the construction of $J$.}
\label{fig:chbij1}
\end{figure}
\end{center}
\end{example} 

\begin{example}
Figure~\ref{fig:chbij2} gives another example of the construction
where $a(w)-d(w)$ is odd for all $w$. Here $W=\{1,2,\ldots,18\}$,
the $f$-chains are shown on the left of the figure,
$I=\{1,5,10,15\}$, $T=\{4,9,14,18\}$, and
$h:T\rightarrow I$ sends $4$ to $15$, $9$ to $1$, 
$14$ to $5$, and $18$ to $10$. We picked $w_0=1$,
but choosing $w_0=10$ instead would lead to the same $J$.  
\begin{center}
\begin{figure}
\epsfig{file=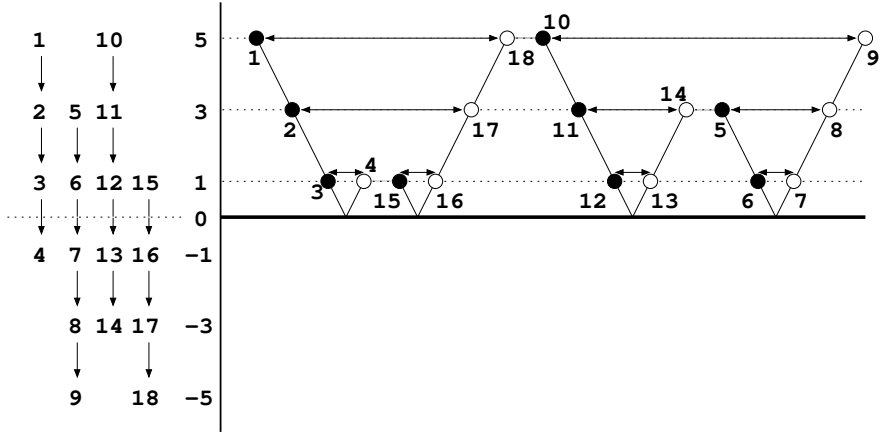,scale=0.75}
\caption{Another example of the construction of $J$.}
\label{fig:chbij2}
\end{figure}
\end{center}
\end{example}

\begin{rem}\label{rem:reattach-chains}
In the situation where $a(w)\geq d(w)$ for all $w\in I$,
one can use the $f$-chains to give a simpler bijective proof
of joint symmetry. For, in this situation, every $f$-chain
must touch the ``midline'' where $a(w)-d(w)=0$ (shown as a 
dotted line on the left in Figures~\ref{fig:chbij1} and~\ref{fig:chbij2}).
We can break all the $f$-chains at the midline and rearrange
the bottom halves (using the bijection $h:T\rightarrow I$)
to create new chains, each of which is symmetric about the midline.
More precisely, the top half of the $f$-chain starting at $h(w)$
is reattached at the midline to the bottom half of the $f$-chain ending 
at $w$, for all $w\in T$. For instance, in Figure~\ref{fig:chbij2},
the new chains would be:
\begin{eqnarray*}
1\longrightarrow 
2\longrightarrow 
3\longrightarrow 
7\longrightarrow 
8\longrightarrow 
9, & &
15\longrightarrow 4, \\ 
10\longrightarrow 
11\longrightarrow 
12\longrightarrow 
16\longrightarrow 
17\longrightarrow 
18, & &
5\longrightarrow 
6\longrightarrow 
13\longrightarrow 
14.  
\end{eqnarray*}
For these new symmetric chains, we can define $J:W\rightarrow W$ by
sending each $w\in W$ to the unique object $w'$ on the same chain as $w$
that has $a(w')=d(w)$ and $d(w')=a(w)$.  
\end{rem}

\begin{rem} Mikhail Mazin gives another combinatorial proof of Theorem \ref{thm:chain-symm} and we briefly describe it here. Note that it does not give a canonical bijection as above.

 First of all, one can assume that $W$ is homogeneous, i.e. the total degree $a(w)+d(w)$ is fixed for $w\in W$. Furthermore, we can think of $W$ as a multiset of integers $d(w)-a(w)$.  Then it suffices to prove the following assertion.
 
\emph{Assertion}: if $W$ is a multiset of integers with a partition into chains where the chain map adds 2, and $I=-T$ where $I$ is the multiset of initial numbers of these chains and $T$ is the multiset of the terminal numbers of them, then we assert that $W=-W$, i.e. $W$ is symmetric.

We say that two chains are overlapping if they have common elements and neither one contains the other. It is easy to adjust the chains into non-overlapping ones. 
We claim that all chains are either symmetric or come in couples: $i\to\cdots\to j$ and $-j\to\cdots\to -i$, therefore $W=-W$.
To prove the claim, take a chain $i\to\cdots\to j$ of minimal length. Then there are no initial and no terminal numbers between $i$ and $j$. Note that $-i$ is terminal and $-j$ is initial by the symmetry $I=-T$. Therefore, there are two possibilities: 
(a) if $i<0<j$, then $i=-j$ and the chain is symmetric; 
(b) if not, then there exists a chain which starts at $-j$ and ends at $-i$, otherwise the chain starting at $-j$ and the one ending at $-i$ are overlapping.
The claim easily follows from here.
\end{rem}

\section{Chain Maps for Higher $q,t$-Catalan Polynomials}
\label{sec:chain-mcat}

\subsection{The Chain Conjecture for $m$-Dyck Words}
\label{subsec:chain-conj}

Throughout this section, we fix $m,n\in\N^+$ and let $W=W_n^{(m)}$.
We propose to prove the joint symmetry of $C^{(m)}_n(q,t)$ 
using the methods of Section~\ref{sec:symm-chains}. More precisely, 
we make the following conjecture.

\begin{conj}\label{reduced_conj}\label{conj:reduced}
Let $T=T_n^{(m)}$ be the set of $m$-Dyck words of length $n$ with $\gamma_1=0$, 
namely, $T=\{(0,0,\gamma_2, \ldots,\gamma_{n-1})\in W\}$. 
There exists a set $I=I_n^{(m)}\subseteq W$ and a bijection
$f:W\setminus T\rightarrow W\setminus I$ such that 
\begin{equation}\label{eq:stat-effect}
 \area(f(w))=\area(w)-1\mbox{ and }\dinv_m(f(w))=\dinv_m(w)+1
  \mbox{ for all $w\in W\setminus T$,} 
\end{equation} 
and $C_T(q,t)=C_I(t,q)$ (using the statistics $\area$ and $\dinv_m$).  
\end{conj}

We will prove this conjecture for all $m\geq 1$ and all $n\leq 4$.
Using Theorem~\ref{thm:chain-symm}, we can then conclude that
$C^{(m)}_n(q,t)=C^{(m)}_n(t,q)$ for these choices of $m$ and $n$.
There are two main difficulties in proving the conjecture for larger $n$.
First, even if we can guess what the map $f$ should be (as we can
for $n=5$), it may be hard to characterize the set $I$ of initial objects 
for the $f$-chains and hence to prove $C_T(q,t)=C_I(t,q)$.  
Second, it does not seem possible to give a single unified formula
for the map $f$. Even for $n=4$, we will need to build $f$ by pasting 
together ``partial chain maps'' $f_0$ (defined in \S\ref{subsec:f0}) 
and $f_1$ (defined in \S\ref{subsec:prove-n=4}) defined on
domains smaller than $W\setminus T$. 

\subsection{The Partial Chain Map $f_0$.}
\label{subsec:f0}

This subsection defines, for any fixed $n$ and $m$,
a set $A_0\subseteq W\setminus T$,
a set $B_0\subseteq W$, and a bijection $f_0:A_0\rightarrow B_0$
that satisfies the formulas in~\eqref{eq:stat-effect} for 
all $w\in A_0$. Informally speaking, $f_0$ is the ``default version''
of the chain map $f$, but it can only be applied to the
$m$-Dyck words in $A_0$. For $n<4$, $A_0$ is the whole set $W\setminus T$, so $f_0$ is already
sufficient to build all the chains needed to prove joint symmetry (see \S\ref{subsec:prove-n=2} and \S\ref{subsec:prove-n=3}). 
For $n=4$, however, we will need to glue $f_0$ with another map $f_1$ as discussed in \S\ref{subsec:prove-n=4}.

\begin{defn}\label{defn:f0} 
For any $\gamma\in W$, let 
$r=r(\gamma)$ be the minimum index $i\in\{2,3,\ldots,n-1\}$
with $\gamma_i-\gamma_{i-2}\leq m$, or $n$ if no such index exists.
Let $A_0$ be the set of $\gamma\in W$ such that
$\gamma_{r-1}-1\leq \gamma_{n-1}+m$ and $\gamma_1>0$
(note $A_0\subseteq W\setminus T$).
Define a map $f_0$ with domain $A_0$ by
$$f_0(\gamma)=(\gamma_0,\gamma_1,\dots,\gamma_{r-2},
\gamma_r,\gamma_{r+1},\dots,\gamma_{n-1},\gamma_{r-1}-1)\in W, \quad \forall \gamma\in A_0.$$
\end{defn}
We assert that $\gamma_{r-1}>0$ if and only if 
$\gamma\in W\setminus T$ (i.e., $\gamma_1>0$).
Indeed,  the equivalence is obvious in the case $r=2$ since $\gamma_{r-1}=\gamma_1$. 
For the case $r\ge3$, we have both $\gamma_{r-1}>0$ and $\gamma_1>0$: indeed,
by the definition of $r$ we have
$\gamma_{r-1}>\gamma_{r-3}+m>0$, while $\gamma_1\ge \gamma_2-m>\gamma_0+m-m=0$.

\begin{lem}\label{lem:f0-dinv} For all $\gamma\in A_0$,
$$\area(f_0(\gamma))=\area(\gamma)-1 \text{ and }
\dinv_m(f_0(\gamma))=\dinv_m(\gamma)+1.$$
\end{lem} 
\begin{proof}
The first equality is immediate.  %Let us focus on the second equality.  
By definition of $\dinv_m$,
\small
\begin{eqnarray*}
\dinv_m(\gamma) &=& 
\sum_{\substack{i<j \\ i\neq r-1\neq j}} \msc(\gamma_i-\gamma_j)
+\sum_{i<r-1} \msc(\gamma_i-\gamma_{r-1})
+\sum_{r-1<i} \msc(\gamma_{r-1}-\gamma_i), \\
\dinv_m(f_0(\gamma)) &=&
\sum_{\substack{i<j \\ i\neq r-1\neq j}} \msc(\gamma_i-\gamma_j)+
\sum_{i\neq r-1} \msc(\gamma_i-\gamma_{r-1}+1).
\end{eqnarray*}
\normalsize
Let us simplify the second summation in the second equation.
Since $\msc(x)=\msc(1-x)$ for any $x\in\mathbb{R}$, this summation equals
$\sum_{i\neq r-1} \msc(\gamma_{r-1}-\gamma_i)$. 
Thus to show $\dinv_m(f_0(\gamma))=\dinv_m(\gamma)+1$, we need only show
$\sum_{i<r-1} \msc(\gamma_{r-1}-\gamma_i)
=1+\sum_{i<r-1} \msc(\gamma_i-\gamma_{r-1})$.  
Since $\gamma_{r-2}-\gamma_{r-3}\le m$ and $\gamma_{r-1}-\gamma_{r-3}>m$,  
we have $\gamma_{r-2}< \gamma_{r-1}$ (note that the inequality holds if $r=2$
since $\gamma\not\in T$).  For every $i<r-2$, 
$$\gamma_i+m<\gamma_{i+2}<\gamma_{i+4}<\cdots<(\gamma_{r-2} \textrm{  or } \gamma_{r-1})\le\gamma_{r-1},$$
therefore $\gamma_{r-1}-\gamma_i>m$, and 
$$\msc(\gamma_{r-1}-\gamma_i)=0=\msc(\gamma_{i}-\gamma_{r-1}).$$ 
For $i=r-2$, $0< \gamma_{r-1}-\gamma_{r-2}\leq m$, hence
$$\msc(\gamma_{r-1}-\gamma_{r-2})=\msc(\gamma_{r-1}-\gamma_{r-2}+1)+1=\msc(\gamma_{r-2}-\gamma_{r-1})+1.$$
Summing up, we conclude that
$$\dinv_m(f_0(\gamma))=1+\sum_{i<j} \msc(\gamma_i-\gamma_j)=\dinv_m(\gamma)+1.
\qedhere $$
\end{proof}

\begin{defn}\label{defn:g0} 
For any $\gamma\in W$, let 
$r'=r'(\gamma)$ be the minimum index $i\in\{2,3,\ldots,n\}$
with $\gamma_{i-2}\geq \gamma_{n-1}+1-m$, or $0$ if no such index exists.
Let $B_0$ be the set of $\gamma\in W$ such that
$r'(\gamma)>0$, $\gamma_{r'-1}\leq\gamma_{n-1}+1+m$,
and $\gamma_i-\gamma_{i-2}>m$ for $2\leq i\leq r'-2$.  
Define a map $g_0$ with domain $B_0$ by setting 
$$g_0(\gamma)=(\gamma_0,\dots,\gamma_{r'-2},
\gamma_{n-1}+1,\gamma_{r'-1},\dots,\gamma_{n-2})\in W,\quad 
\forall\gamma\in B_0.$$  
\end{defn}

\begin{lem}\label{lem:f0-bij} 
The map $f_0$ is a bijection from $A_0$ to $B_0$ with inverse $g_0$.  
\end{lem}
\begin{proof} 
Take any $\gamma\in A_0$. In the proof of Lemma \ref{lem:f0-dinv}, 
we showed that $\gamma_{r-1}-\gamma_i>m$ for $i<r-2$. On the other hand,
$\gamma_{r-1}-\gamma_{r-2}\le m$. Let $\gamma'=f_0(\gamma)$; 
then $r$ is the smallest integer such that $\gamma'_{r-2}\ge\gamma'_{n-1}+1-m$ 
and $\gamma'_i-\gamma'_{i-2}>m$ for $2\le i\le r-2$. 
So $f_0(\gamma)\in B_0$, $r=r'(f_0(\gamma))$, 
and hence Definition~\ref{defn:g0} gives
$$\aligned 
g_0\circ f_0(\gamma)&=g_0((\gamma_0,\dots,\gamma_{r-2},\gamma_r,\dots,\gamma_{n-1},\gamma_{r-1}-1))\\
&=
(\gamma_0,\dots,\gamma_{r-2},(\gamma_{r-1}-1)+1,\gamma_r,\dots,\gamma_{n-1})
=\gamma.
\endaligned
$$
Similarly, we can show that $g_0$ maps $B_0$ into $A_0$
and $f_0\circ g_0(\gamma)=\gamma$ for any $\gamma\in B_0$, 
completing the proof that $f_0$ and $g_0$ are mutually inverse bijections.  
\end{proof}

\section{Proof of Joint Symmetry for $n\leq 4$}
\label{sec:proof-jsymm}

In this section, we prove the chain conjecture~\ref{reduced_conj}, 
and hence the joint symmetry of $C^{(m)}_n(q,t)$, for all positive 
integers $m,n$ with $n\leq 4$. First note that these conjectures hold
when $n=1$ because $C_1^{(m)}(q,t)=1$ for every $m$. In the next three 
subsections we shall settle the cases $n=2,3,4$, respectively.

\subsection{The Case $n=2$.}
\label{subsec:prove-n=2}
We have $W=\W^{(m)}_2=\{(0,i):0\leq i\leq m\}$
and $C^{(m)}_2(q,t)=\sum_{i=0}^m q^it^{m-i}$, which is
evidently jointly symmetric. Nevertheless,
let us see what our approach does when $n=2$. We have
$T=\{(0,0)\}$, $r(\gamma)=2$ for all $\gamma\in W$,
$A_0=\{(0,i):1\leq i\leq m\}=W\setminus T$,
$B_0=\{(0,i):0\leq i<m\}$, and % CHECK THIS
$I=W\setminus B_0=\{(0,m)\}$.
We take the map $f:W\setminus T\rightarrow W\setminus I$
to be the map $f_0:A_0\rightarrow B_0$ given by
$f_0((0,i))=(0,i-1)$ for $1\leq i\leq m$, which
has the correct effect on $\area$ and $\dinv_m$.
The unique map $h:T\rightarrow I$ gives a bijective proof
of $C_T(q,t)=C_I(t,q)$ since
\[ C_T(q,t)=q^{\area((0,0))}t^{\dinv_m((0,0))}
     =t^m  =t^{\area((0,m))}q^{\dinv_m((0,m))} = C_I(t,q). \]
So $C^{(m)}_2(q,t)=C^{(m)}_2(t,q)$ by Theorem~\ref{thm:chain-symm}.

\subsection{The Case $n=3$.}
\label{subsec:prove-n=3}

Let $W=\W^{(m)}_3$. We have $T=\{(0,0,i): 0\leq i\leq m\}$ and
$A_0=\{\gamma\in W: \gamma_1>0\}=W\setminus T$,
so we can take $f=f_0:A_0\rightarrow B_0$ and 
$I=W\setminus B_0=\{(0,i,i+m): 0\leq i\leq m\}$.
For $n=3$, the definition of $f$ can be rephrased as follows: 
$$f(\gamma)=
\begin{cases}
(\gamma_0,\gamma_2,\gamma_1-1), \text{ if }\gamma_2\leq m;\\
(\gamma_0,\gamma_1,\gamma_2-1), \text{ if }\gamma_2> m.\\
\end{cases}
$$
We already proved that $f$ is a bijection (Lemma~\ref{lem:f0-bij})
having the correct effect on $\area$ and $\dinv_m$ (Lemma~\ref{lem:f0-dinv}).
We need to prove that $C_T(q,t)=C_I(t,q)$.  Define 
$$
\aligned
&g:I\setminus\{(0,0,m)\} \longrightarrow I\setminus\{(0,m,2m)\}, \quad
\gamma\mapsto (\gamma_0, \gamma_1-1,\gamma_2-1);\\
&g': T\setminus\{(0,0,m)\} \longrightarrow T\setminus\{(0,0,0)\}, \quad
\gamma\mapsto (0, 0, \gamma_2+1).
\endaligned
$$
Then $g$ increases $\dinv_m$ by 1 and decreases $\area$ by 2,
whereas $g'$ increases $\area$ by 1 and decreases $\dinv_m$ by 2. 
The set $I$ consists of a single
$g$-chain of $m+1$ objects starting at $(0,m,2m)$ and ending at $(0,0,m)$. So
$$
\aligned
 C_I(q,t)&=q^{\area((0,m,2m))}t^{\dinv_m((0,m,2m))} 
  (1+q^{-2}t+q^{-4}t^2+\cdots+q^{-2m}t^{m})\\
  &=q^{3m}t^0+q^{3m-2}t^1+q^{3m-4}t^2+\cdots+q^{m}t^{m}.\\
\endaligned
$$
Similarly, 
$C_T(q,t)=q^{0}t^{3m}+q^{1}t^{3m-2}+q^{2}t^{3m-4}+\cdots+q^{m}t^{m}$.
Hence $C_T(q,t)=C_I(t,q)$. Note that every monomial in $C_I(q,t)$ has 
a different total degree. It follows that 
$$
C_3^{(m)}(q,t)= \sum_{j=0}^m(q^{3m-2j}t^j+q^{3m-2j-1}t^{j+1}+\cdots+q^jt^{3m-2j}),
$$
so $C_3^{(m)}(q,t)$ is
the sum of all monomials $q^jt^k$ where $(j,k)$ runs through all lattice points that are either inside or on the boundary of the triangle with vertices $(0,3m)$, $(m,m)$, $(3m,0)$ (Figure \ref{fig:C3tri}).
In particular, for $n=3$, there is a unique involution on $W$ giving a 
bijective proof of joint symmetry.

\begin{center}
\begin{figure}
\epsfig{file=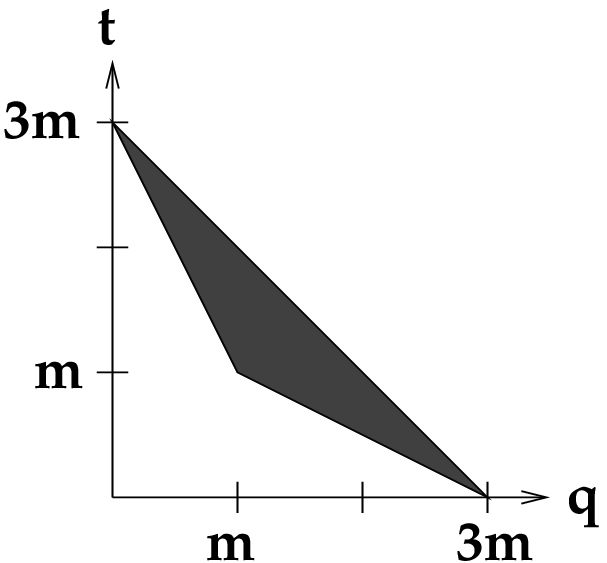,scale=.7}
\hspace{.5in}
\epsfig{file=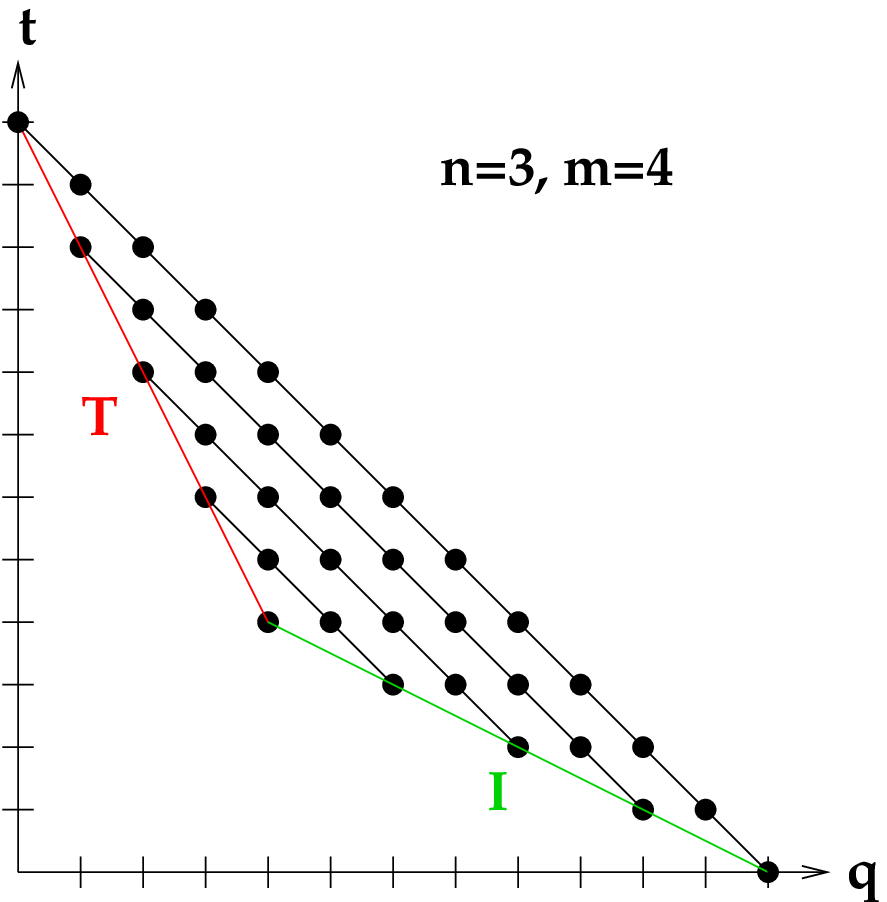,scale=.65}
\caption{Left: the triangle that contains all $(j,k)$ for monomials $q^jt^k$ in $C_3^{(m)}(q,t)$.
Right: the special case when $m=4$ (the lattice points on the green line give monomials in $C_I(q,t)$; the lattice points on the red line give monomials in $C_T(q,t)$. The black lines illustrate $f$-chains).}
\label{fig:C3tri}
\end{figure}
\end{center}

\subsection{The Case $n=4$.}
\label{subsec:prove-n=4}
Let $W=\W^{(m)}_4$. We define a partial chain map $f_1$ and glue it with $f_0$ to obtain a bijective map $f$ that satisfies Conjecture \ref{reduced_conj}.

\begin{defn}\label{def:f1}
Let $A_1=\{(0,\gamma_1,\gamma_2,\gamma_3)\in\W: \gamma_2-\gamma_3>m+1\}$.
(Note that $A_1\subseteq W\setminus T$ and $A_0\cap A_1=\emptyset$.)
Define a map $f_1$ with domain $A_1$ by 
$$f_1((0,\gamma_1,\gamma_2,\gamma_3))
=(0,\gamma_3+1,\gamma_1-1,\gamma_2-1)\in W.$$
Let $B_1=f_1(A_1)$. 
Define $A=A_0\cup A_1$, $B=B_0\cup B_1$, and
$f:A\rightarrow B$ by $f(\gamma)=f_0(\gamma)$ for $\gamma\in A_0$
and $f(\gamma)=f_1(\gamma)$ for $\gamma\in A_1$.
\end{defn}

\begin{lem}\label{lem:gamma-A} 
$\W\setminus A=\{(0,0,\gamma_2,\gamma_3)\in\W\}=T.$
\end{lem}
\begin{proof}
First, observe that $\W\setminus A_0$ 
contains the following three classes of elements (which correspond to $r=2,3,4$, respectively):
\begin{itemize}
\item[(a)] $\gamma_2\le m$, $\gamma_1=0$; 
\item[(b)] $\gamma_2>m$, $\gamma_3-\gamma_1\le m$, and $\gamma_2-1>\gamma_3+m$; 
%was: \item[(c)] $\gamma_2>0$, $\gamma_3-\gamma_1>m$, and $\gamma_3=0$.  
\item[(c)] $\gamma_2>m$, $\gamma_3-\gamma_1>m$, and $\gamma_3=0$.
\end{itemize} 
(Note that for $r=3$ we do not need to include the case $\gamma_2>m$, $\gamma_3-\gamma_1\le m$, and $\gamma_1=0$, since $\gamma_1\ge \gamma_2-m>0$.)
Elements satisfying (b) are exactly those in $A_1$. No elements satisfy (c). 
Thus  $\W \setminus A$ contains elements only in (a), i.e., those of the form 
$(0,0,\gamma_2,\gamma_3)$.
\end{proof}

\begin{lem}
$B_0\cap B_1=\emptyset$, so $f:A\to B$ is a bijection.
\end{lem}
\begin{proof}
Assume, to the contrary, that there is a $\gamma\in A_1$ such that  
$f(\gamma)=\gamma'\in B_0$. Thus $\gamma_1'=\gamma_3+1$, 
$\gamma_2'=\gamma_1-1$, $\gamma_3'=\gamma_2-1$. Let $r'=r'(\gamma')$ be
defined as in Definition~\ref{defn:g0}.  
Note $r'$ can be $2$, $3$, or $4$. 
Below we show that all three cases give a contradiction.
\begin{itemize}
\item[(i)] If $r'=2$, then $0\ge \gamma_3'+1-m$ by the definition of $r'$, hence $\gamma_2\le m$. But $\gamma_2-\gamma_3>m+1$ by the definition of $A_1$. Thus $\gamma_2>m$, a contradiction.
\item[(ii)] If $r'=3$, then $\gamma'_1\ge \gamma'_3+1-m$ by the definition of $r'$, which implies $\gamma_2-\gamma_3\le m+1$ and again contradicts the definition of $A_1$.  
\item[(iii)] If $r'=4$, then $\gamma'_2=\gamma'_2-\gamma'_0>m$ 
by the definition of $B_0$. 
Thus $\gamma_1>m+1$ which contradicts the condition that $\gamma_1\le m$.  
\end{itemize}
Since $f_1$ is evidently an injection, we conclude that $f$ is a bijection, 
whose inverse can be defined by $f^{-1}(\gamma)=f_0^{-1}(\gamma)$ for
$\gamma\in B_0$ and $f^{-1}(\gamma)=(0,\gamma_2+1,\gamma_3+1,\gamma_1-1)$ 
for $\gamma\in B_1$.
\end{proof}

\begin{lem}\label{lem: dinv of tilde f} For every $\gamma\in A$,
$$\area(f(\gamma))=\area(\gamma)-1 \quad \text{and} 
\quad \dinv_m(f(\gamma))=\dinv_m(\gamma)+1.$$
\end{lem} 
\begin{proof}
Because of Lemma~\ref{lem:f0-dinv}, 
we need only to consider $\gamma\in A_1$. The first equality is evident. 
The left hand side of the second equality is 
\begin{multline*}
 \msc(-\gamma_1+1)+\msc(-\gamma_2+1)+\msc(-\gamma_3-1)
+\msc(\gamma_1-\gamma_2)\\+\msc(\gamma_3-\gamma_1+2)+\msc(\gamma_3-\gamma_2+2).
\end{multline*}
Write each summand in terms of $\msc(\gamma_i-\gamma_j)$ for $i<j$ 
(and recall that $\gamma_0=0$):
\begin{itemize}
%was: \item[(i)] $0\le\gamma_1\le m\Rightarrow \msc(-\gamma_1+1)
% =\msc(-\gamma_1)+1$.  
\item[(i)] $0<\gamma_1\le m\Rightarrow \msc(-\gamma_1+1)=\msc(-\gamma_1)+1$.  
\item[(ii)] $\gamma_2-\gamma_3>m+1\Rightarrow \gamma_2\ge m+2\Rightarrow \msc(-\gamma_2+1)=\msc(-\gamma_2)=0$.  
\item[(iii)] $(\gamma_2-\gamma_3>m+1)$ and $(\gamma_2\le 2m)\Rightarrow
0\leq \gamma_3<m-1\Rightarrow \msc(-\gamma_3-1)=\msc(-\gamma_3)-1$.  
\item[(iv)] $\msc(\gamma_1-\gamma_2)$ is unchanged.  
\item[(v)] $m\ge\gamma_1\ge\gamma_2-m>\gamma_3+1\Rightarrow 1< \gamma_1-\gamma_3\le m\Rightarrow \msc(\gamma_3-\gamma_1+2)=\msc(\gamma_1-\gamma_3-1)=\msc(\gamma_1-\gamma_3)+1$.  
\item[(vi)] $\gamma_2-\gamma_3>m+1\Rightarrow \msc(\gamma_3-\gamma_2+2)=\msc(\gamma_2-\gamma_3-1)=\msc(\gamma_2-\gamma_3)=0$.
\end{itemize}
Summing up (i)--(vi), we obtain the second equality.
\end{proof}

\begin{lem}\label{lem:W-B union}
{\rm(1)} $B_1=\{\gamma\in\W: \gamma_1\ge 1, \gamma_2\le m-1, 
\gamma_3-\gamma_1\ge m\}$; 

{\rm(2)} $I=\W\setminus B$ is the disjoint union $D_1\cup D_2\cup D_3$, where:
$$\aligned 
&D_1=\{(0,\gamma_1, \gamma_2,\gamma_2+m):
0<\gamma_1\le m< \gamma_2\le\gamma_1+m\},\\
&D_2=\{(0,\gamma_1,m,\gamma_3):
0\le \gamma_1\le m, \gamma_1+m\le\gamma_3\le 2m\},\\
&D_3=\{(0,0,\gamma_2,\gamma_3): 
0\le\gamma_2< m\le \gamma_3\le\gamma_2+m\}.
\endaligned$$
\end{lem}
\begin{proof} 
The proof of (1) is routine. For (2), we first find out $\W\setminus B_0$.
By a case-by-case study for $r'=2, 3, 4$, it is straightforward to check 
that $B_0$ contains those $\gamma$ with
$\gamma_3\le m-1$ (in the case $r'(\gamma)=2$), 
those with $m\le\gamma_3\le\gamma_1+m-1$ (in the case $r'(\gamma)=3$), and 
those with $\gamma_1+m\le \gamma_3\le\gamma_2+m-1$ and $\gamma_2>m$ 
(in the case $r'(\gamma)=4$). Hence
$\W \setminus B_0$ consists of $\gamma\in\W$ that satisfy
\begin{equation}\label{W-B 1}
\gamma_1+m\le \gamma_3\le\gamma_2+m-1 \textrm{ and }\gamma_2\le m
\end{equation}
or
\begin{equation}\label{W-B 2}
\gamma_3\ge\max(\gamma_1,\gamma_2)+m.
\end{equation}
Note that (\ref{W-B 2}) is equivalent to 
$\gamma_1\le\gamma_2$ and $\gamma_3=\gamma_2+m$, 
because $\gamma_3\leq \gamma_2+m$.
Therefore $\W\setminus B$ consists of three classes of elements:

$(0,\gamma_1, \gamma_2,\gamma_2+m)$ where $\gamma_1\le m< \gamma_2$,

$(0,0,\gamma_2,\gamma_3)$ where $\gamma_2< m\le \gamma_3$, and

$(0,\gamma_1,m,\gamma_3)$ where $\gamma_1+m\le\gamma_3$,

\noindent which gives (2).
\end{proof}

\begin{rem}
One can check that $f$ can be defined piecewise as follows:
$$
\aligned
& \text{ if }\gamma_2\leq m\text{ then } f(\gamma)= (\gamma_0,\gamma_2,\gamma_3,\gamma_1-1);\\
& \text{ if }\gamma_2>m\text{ then }\\
& \quad \left| 
\aligned
& \text{ if }\gamma_3-\gamma_1> m\text{ then } f(\gamma)= (\gamma_0,\gamma_1,\gamma_2,\gamma_3-1); \\
& \text{ if }\gamma_3-\gamma_1\leq m\text{ then }\\
& \quad \left|  
\aligned 
 &  \text{ if }\gamma_2-\gamma_3> m+1\text{ then }f(\gamma)= (\gamma_0,\gamma_3+1,\gamma_1-1,\gamma_2-1); \\
 &  \text{ if }\gamma_2-\gamma_3\leq  m+1\text{ then } f(\gamma)= (\gamma_0,\gamma_1,\gamma_3,\gamma_2-1).\\
\endaligned\right.\\
\endaligned\right.\\
\endaligned
$$
\end{rem}
Now to finish the proof of Conjecture~\ref{reduced_conj} for $n=4$, we shall show the following: 
\begin{equation}\label{eq:head-tail symmetry}
\sum_{\gamma\in\W \setminus A} q^{\area(\gamma)}t^{\dinv_m(\gamma)}
=\sum_{\gamma\in\W \setminus B} t^{\area(\gamma)}q^{\dinv_m(\gamma)}.
\end{equation}
By Lemma \ref{lem:gamma-A}, $\W \setminus A$ consists of $\gamma=(0,0,\gamma_2,\gamma_3)\in\W$. Since $\gamma_2\leq m$, we can write $\W\setminus A$ as the disjoint
union $D_1'\cup D_2'\cup D'_3$, where 
\begin{eqnarray*} 
D'_1 &=& \{\gamma\in \W\setminus A: \gamma_2>\gamma_3\},  \\
D'_2 &=& \{\gamma\in \W\setminus A: \gamma_2\le \gamma_3\le m\},\\
D'_3 &=& \{\gamma\in \W\setminus A:  \gamma_3>m\}.
\end{eqnarray*}
Then $\dinv_m(\gamma)
=\msc(0)+2\msc(-\gamma_2)+2\msc(-\gamma_3)+\msc(\gamma_2-\gamma_3)$ is equal to
\small
\begin{equation}\label{h+D'123}
\begin{cases}
m+2(m-\gamma_2)+2(m-\gamma_3)+(m+1-\gamma_2+\gamma_3)=6m+1-3\gamma_2-\gamma_3,& \textrm{ if } \gamma\in D'_1;\\
m+2(m-\gamma_2)+2(m-\gamma_3)+(m+\gamma_2-\gamma_3)=6m-\gamma_2-3\gamma_3,&\textrm{ if }\gamma\in D'_2;\\
m+2(m-\gamma_2)+0+(m+\gamma_2-\gamma_3)=4m-\gamma_2-\gamma_3,& \textrm{ if } \gamma\in D'_3.\\
\end{cases}
\end{equation}
\normalsize
On the other hand, thanks to
Lemma~\ref{lem:W-B union}, $\dinv_m(\gamma)$ is equal to 
\small
\begin{equation}\label{h+D123}
\begin{cases}
(m-\gamma_1)+0+0+(m+\gamma_1-\gamma_2)+0+0=2m-\gamma_2, 
&\textrm{ if } \gamma\in D_1;\\
(m-\gamma_1)+0+0+(m+\gamma_1-m)+0+(m+m-\gamma_3)=3m-\gamma_3,
&\textrm{ if } \gamma\in D_2;\\
m+2(m-\gamma_2)+0+(m+\gamma_2-\gamma_3)=4m-\gamma_2-\gamma_3, 
&\textrm{ if } \gamma\in D_3.\\
\end{cases}
\end{equation}
\normalsize
Therefore, the assertion (\ref{eq:head-tail symmetry}) follows from the following lemma.  
\begin{lem}
\begin{equation}\label{eq:head-tail symmetry-12}
\sum_{\gamma\in D_1\cup D_2} q^{\area(\gamma)}t^{\dinv_m(\gamma)}
=\sum_{\gamma\in D'_1\cup D'_2} t^{\area(\gamma)}q^{\dinv_m(\gamma)};
\end{equation}
\begin{equation}\label{eq:head-tail symmetry-3}
\sum_{\gamma\in D_3} q^{\area(\gamma)}t^{\dinv_m(\gamma)}
=\sum_{\gamma\in D'_3} t^{\area(\gamma)}q^{\dinv_m(\gamma)}.
\end{equation}
\end{lem}

\begin{proof}
The equality (\ref{eq:head-tail symmetry-3}) follows immediately from the one-to-one correspondence from $D'_3$ to $D_3$ that sends $(0,0,\gamma_2,\gamma_3)$ to $(0,0,2m-\gamma_3,2m-\gamma_2)$. 
Alternatively, both sides can be shown to be equal to 
\begin{equation}\label{eq:sum'}
\sum_{u=0}^{m-1}\sum_{v=m}^{u+m}q^{u+v}t^{4m-u-v}.
\end{equation}

Next we prove (\ref{eq:head-tail symmetry-12}).  Define 
\small
$$g : (D_1\cup D_2)\setminus  \{(0,0,m,\gamma_3)\in \W \, :\,  \gamma_3\geq m \} \longrightarrow (D_1\cup D_2)\setminus\{(0,\gamma_1,\gamma_1+m,\gamma_1+2m)\in \W \}$$
\normalsize
by
$$
g(\gamma)=\left\{ 
\begin{array}{ll}
(0, \gamma_1,\gamma_2-1,\gamma_3-1), & \text{ if }\gamma\in D_1;\\
(0, \gamma_1-1,\gamma_2,\gamma_3-1), & \text{ if }\gamma\in D_2.
\end{array}
 \right.
$$  
Note that $g$ is a bijection with inverse defined as
$$
g^{-1}(\gamma)=\left\{ 
\begin{array}{ll}
(0, \gamma_1,\gamma_2+1,\gamma_3+1), & \text{ if } m\le \gamma_2 <\gamma_1+m, \gamma_3=\gamma_2+m;\\
(0, \gamma_1+1,\gamma_2,\gamma_3+1), & \text{ if }\gamma_1<m, \gamma_1+m\le \gamma_3<2m.
\end{array}
 \right.
$$%Easy to check that the domain of g^{-1} is the codomain of g.
Using the formula (\ref{h+D123}) to compute $\dinv_m$, it is straightforward to verify that
\begin{equation}\label{eq:two sums}
\aligned 
&\sum_{\gamma\in \{(0,0,m,\gamma_3)\in \W \, :\,  \gamma_3\geq m \}} 
q^{\area(\gamma)}t^{\dinv_m(\gamma)}          =       q^{2m}t^{2m} + q^{2m+1}t^{2m-1} +\cdots + q^{3m}t^{m},\\
&\sum_{\gamma\in \{(0,\gamma_1,\gamma_1+m,\gamma_1+2m)\in \W \} } q^{\area(\gamma)}t^{\dinv_m(\gamma)}          = 
q^{6m}t^{0} + q^{6m-3}t^{1} + \cdots + q^{3m}t^{m},
\endaligned
\end{equation}
and that $g$ 
increases $\dinv_m$ by 1 and decreases $\area$ by 2. 
By iterating the map $g$ we obtain maximal sequences $(\gamma, g(\gamma), ..., g^{(\gamma_1)}(\gamma))$ that start from elements in $\{(0,\gamma_1,\gamma_1+m,\gamma_1+2m)\in \W\}$ and end at elements in
$\{(0,0,m,\gamma_3)\in \W \, :\,  \gamma_3\geq m \}$. Moreover, if we define the $(1,2)$-weight of $\gamma$ to be the integer $\area(\gamma)+2\dinv_m(\gamma)$, then each sequence contains elements of the same $(1,2)$-weight. Since the $(1,2)$-weights of those $\gamma$ appearing
 in the first (resp. second) equation of \eqref{eq:two sums} are distinct, the last element of a sequence is determined by the first element. As a result, the left-hand side of \eqref{eq:head-tail symmetry-12} must equal 
\begin{equation}\label{eq:sum}\aligned
\sum_{u=0}^m\sum_{v=u}^{2m-u}q^{6m-u-2v}t^{v}=& (q^{6m}t^{0} + q^{6m-2}t^{1} + q^{6m-4}t^{2} + q^{6m-6}t^{3} + \cdots + q^{2m}t^{2m})\\
+ &(q^{6m-3}t^{1} + q^{6m-5}t^{2} + q^{6m-7}t^{3} + \cdots +   q^{2m+1}t^{2m-1})\\
+ &(q^{6m-6}t^{2} + q^{6m-8}t^{3} + \cdots +   q^{2m+2}t^{2m-2})\\
 &\vdots\\
+ &(q^{3m}t^{m}).
\endaligned
\end{equation} 
On the other hand, define $$g' : (D'_1\cup D'_2)\setminus  \{(0,0,\gamma_2,m)\in \W \} \longrightarrow (D'_1\cup D'_2)\setminus\{(0,0,0,\gamma_3)\in \W \}$$ by
$ g'(\gamma)=(0, 0,\gamma_3+1,\gamma_2)$. Note that $g'$ is a bijection with inverse $g'^{-1}(\gamma)=(0,0,\gamma_3,\gamma_2-1)$. Using (\ref{h+D'123}), we can verify that
\small $$\aligned
&\sum_{\gamma\in \{(0,0,\gamma_2,m)\in \W \}} 
t^{\area(\gamma)}q^{\dinv_m(\gamma)}          =       q^{2m}t^{2m} + q^{2m+1}t^{2m-1} +  q^{2m+2}t^{2m-2} +\cdots + q^{3m}t^{m},\\
&\sum_{\gamma\in \{(0,0,0,\gamma_3)\in \W \}}  t^{\area(\gamma)}q^{\dinv_m(\gamma)}         = 
q^{6m}t^{0} + q^{6m-3}t^{1} + q^{6m-6}t^{2} + \cdots + q^{3m}t^{m}.
\endaligned
$$\normalsize
and that $g'$ increases $\area$ by 1 and decreases $\dinv_m$ by 2. 
For each $\gamma=(0,0,0,\gamma_3)\in \W$, iterating the map $g'$ produces the sequence $\gamma$, $g'(\gamma), ...$, $g'^{(m-\gamma_3)}(\gamma)$ where $g'^{(m-\gamma_3)}(\gamma)$ is in the set $\{(0,0,\gamma_2,m)\in \W \}$.
By a similar argument as above, the right-hand side of \eqref{eq:head-tail symmetry-12} is also equal to \eqref{eq:sum}. This finishes the proof of \eqref{eq:head-tail symmetry-12} and therefore completes the proof of Conjecture~\ref{reduced_conj} for $n=4$.
\end{proof}
\begin{rem}\label{remark:4.2}
We can also prove \eqref{eq:head-tail symmetry-12} directly by showing that
\begin{equation}\label{eq:head-tail symmetry-12D}
\sum_{\gamma\in D'_1\cup D'_2} q^{\area(\gamma)}t^{\dinv_m(\gamma)}
=\sum_{(x,y)\in\Delta}q^xt^y=\sum_{\gamma\in D_1\cup D_2} 
t^{\area(\gamma)}q^{\dinv_m(\gamma)},
\end{equation}
where $\Delta$ is the set of lattice points that are either inside or on the boundary of the triangle with vertices $(0,6m)$, $(m,3m)$, $(2m,2m)$ (Figure \ref{fig:C4tri} Left).
\begin{center}
\begin{figure}
\epsfig{file=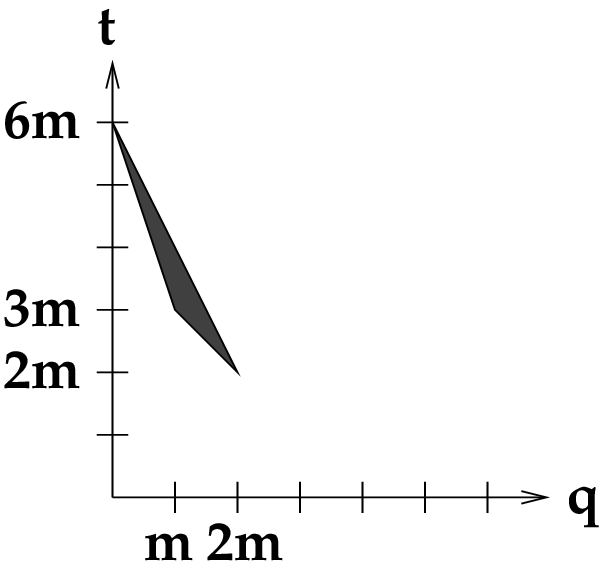,scale=.7}
\hspace{.5in}
\epsfig{file=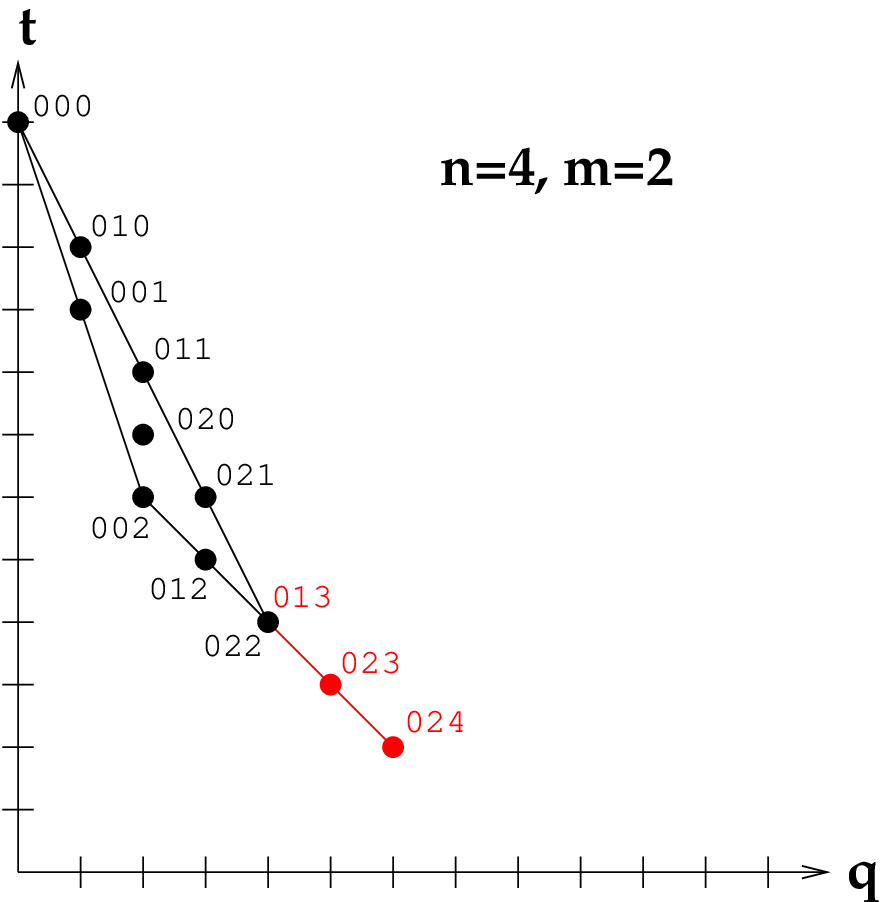,scale=.65}
\caption{Left: the triangle that contains $\Delta$ defined in Remark \ref{remark:4.2}. Right: the support of $C_T(q,t)$ in the case $n=4, m=2$ as discussed in Example \ref{ex:chain4-2}, where $\Delta=D'_1\cup D'_2$ is the set of the black lattice points, and $D'_3$ consists of the three lattice points on the red line segment.}
\label{fig:C4tri}
\end{figure}
\end{center}

Indeed, note that the set of all points 
inside or on the boundary of the triangle is given by $\{(u+v,6m-u-3v): u,v\in\mathbb{R}, 0\le u\le v\le m\}$. To be a lattice point in this triangle, either $u$ or $u-1/2$ is an integer. The lattice points with $u$ being an integer are in one-to-one correspondence with pairs $(\area(\gamma),\dinv_m(\gamma))$ for $\gamma\in D'_2$ (by letting $\gamma_2=u$, $\gamma_3=v$); the lattice points with $u-1/2$ being an integer are in one-to-one correspondence with pairs $(\area(\gamma),\dinv_m(\gamma))$ for $\gamma\in D'_1$ (by letting $\gamma_2=v+1/2$, $\gamma_3=u-1/2$).
Thus the pairs $(\area(\gamma),\dinv_m(\gamma))$ for $\gamma$ in ${D'_2}$ (resp.~$D'_1$)  form the set $\{(x,y)\in \Delta| x+y\textrm { is even}\}$ (resp.~$\{(x,y)\in \Delta| x+y\textrm { is odd}\}$). 

On the other hand, the set of all
points inside or on the boundary of the triangle can also be defined as $\{(3m-v,m+u+v): u,v\in\mathbb{R}, 0\le u\le 2m, u+m\le v\le u/2+2m\}$. The lattice points therein with $v\le 2m$ are in one-to-one correspondence with pairs $(\dinv_m(\gamma),\area(\gamma))$ for $\gamma\in D_1$ (by letting $\gamma_1=u$, $\gamma_3=v$); the lattice points therein with $v>2m$ are in one-to-one correspondence with pairs $(\dinv_m(\gamma),\area(\gamma))$ for $\gamma\in D_2$ (by letting $\gamma_1=u-v+2m$, $\gamma_2=v-m$).
Thus the pairs $(\dinv_m(\gamma),\area(\gamma))$ for $\gamma$ in ${D_2}$ (resp.~$D_1$)  form the set $\{(x,y)\in \Delta| x\ge m\}$ (resp.~$\{(x,y)\in \Delta| x<m\}$). 

Therefore both equalities in \eqref{eq:head-tail symmetry-12D} hold.
\end{rem}

\begin{example}\label{ex:chain4-2}
Let $n=4$ and $m=2$.
There are fifty-five $2$-Dyck words in $\W=\W^{(2)}_4$. 
Since $\gamma_0$ is always zero,
we use $\gamma_1\gamma_2\gamma_3(q^at^d)$ to denote the Dyck word 
$(\gamma_0,\gamma_1,\gamma_2,\gamma_3)$ with $(\area,\dinv_m)=(a,d)$. 
The $f$-chains for this set are listed below. 

(1) an $f$-chain from $246(q^{12}t^0)$ to $000(q^0t^{12})$ of length 12, 
	
(2) an $f$-chain from $235(q^{10}t^1)$ to $010(q^1t^{10})$ of length 9, 

(3) an $f$-chain from $124(q^7t^2)$ to $020(q^2t^7)$ of length 5, 

(4) an $f$-chain from $123(q^6t^3)$ to $021(q^3t^6)$ of length 3, 

(5) an $f$-chain from $224(q^8t^2)$ to $001(q^1t^9)$ of length 7, 

(6) an $f$-chain from $135(q^9t^1)$ to $011(q^2t^8)$ of length 7, 

(7) six chains of length zero: $024(q^6t^2)$, $023(q^5t^3)$, $022(q^4t^4)$, $012(q^3t^5)$, $002(q^2t^6)$, $013(q^4t^4)$. 

Figure \ref{fig:f chain 4mx4} illustrates the six $f$-chains with 
nonzero lengths, where the length of a chain is defined as the 
number of arrows.

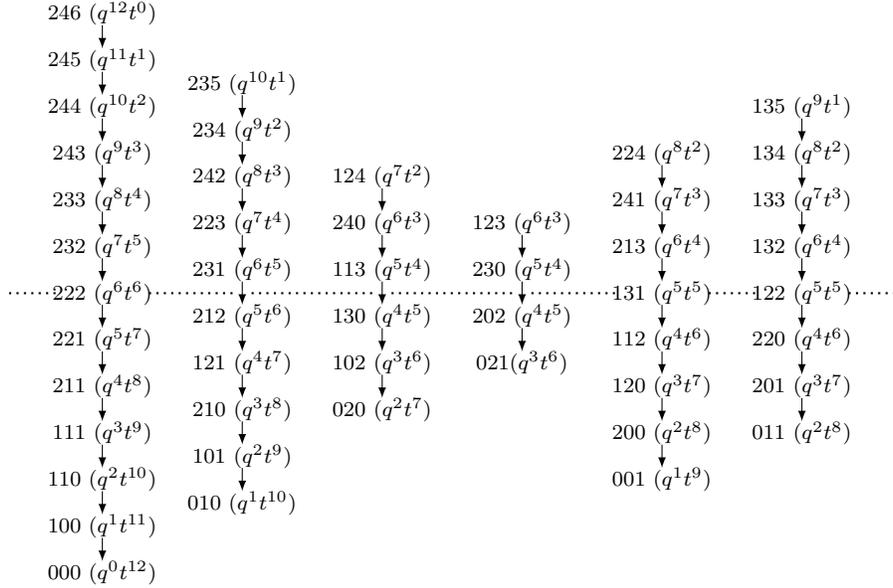
\begin{figure}[t]
\iffalse
\epsfig{file=fig3.eps,scale=1.0}
\fi
{\tiny
$$
\begin{tikzpicture}[scale=.62]
\draw[dotted, thick] (-2,-6) -- (-1,-6) (1,-6) -- (2.8,-6) (3.2,-6) -- (5.8,-6) (6.2,-6) -- (8.8,-6) (9.2,-6) -- (11,-6) (13,-6) -- (14,-6) (16,-6) -- (17,-6);
\begin{scope}[>=latex]
\draw[->] (0,0-.25) -- (0,-1+.25);\draw[->] (0,-1-.25) -- (0,-2+.25);\draw[->] (0,-2-.25) -- (0,-3+.25);\draw[->] (0,-3-.25) -- (0,-4+.25);\draw[->] (0,-4-.25) -- (0,-5+.25);\draw[->] (0,-5-.25) -- (0,-6+.25);\draw[->] (0,-6-.25) -- (0,-7+.25);\draw[->] (0,-7-.25) -- (0,-8+.25);\draw[->] (0,-8-.25) -- (0,-9+.25);\draw[->] (0,-9-.25) -- (0,-10+.25);\draw[->] (0,-10-.25) -- (0,-11+.25);\draw[->] (0,-11-.25) -- (0,-12+.25);
\draw (0,0) node {246 ($q^{12}t^0$)};\draw (0,-1) node {245 ($q^{11}t^1$)};\draw (0,-2) node {244 ($q^{10}t^2$)};\draw (0,-3) node {243 ($q^{9}t^3$)};\draw (0,-4) node {233 ($q^{8}t^4$)};\draw (0,-5) node {232 ($q^{7}t^5$)};\draw (0,-6) node {222 ($q^{6}t^6$)};\draw (0,-7) node {221 ($q^{5}t^7$)};\draw (0,-8) node {211 ($q^{4}t^8$)};\draw (0,-9) node {111 ($q^{3}t^9$)};\draw (0,-10) node {110 ($q^{2}t^{10}$)};\draw (0,-11) node {100 ($q^{1}t^{11}$)};\draw (0,-12) node {000 ($q^{0}t^{12}$)};
\end{scope}
\begin{scope}[shift={(3,-1.5)}, >=latex]
\draw[->] (0,0-.25) -- (0,-1+.25);\draw[->] (0,-1-.25) -- (0,-2+.25);\draw[->] (0,-2-.25) -- (0,-3+.25);\draw[->] (0,-3-.25) -- (0,-4+.25);\draw[->] (0,-4-.25) -- (0,-5+.25);\draw[->] (0,-5-.25) -- (0,-6+.25);\draw[->] (0,-6-.25) -- (0,-7+.25);\draw[->] (0,-7-.25) -- (0,-8+.25);\draw[->] (0,-8-.25) -- (0,-9+.25);
\draw (0,0) node {235 ($q^{10}t^{1}$)};\draw (0,-1) node {234 ($q^{9}t^{2}$)};\draw (0,-2) node {242 ($q^{8}t^{3}$)};\draw (0,-3) node {223 ($q^{7}t^{4}$)};\draw (0,-4) node {231 ($q^{6}t^{5}$)};\draw (0,-5) node {212 ($q^{5}t^{6}$)};\draw (0,-6) node {121 ($q^{4}t^{7}$)};\draw (0,-7) node {210 ($q^{3}t^{8}$)};\draw (0,-8) node {101 ($q^{2}t^{9}$)};\draw (0,-9) node {010 ($q^{1}t^{10}$)};
\end{scope}
\begin{scope}[shift={(6,-3.5)}, >=latex]
\draw[->] (0,0-.25) -- (0,-1+.25);\draw[->] (0,-1-.25) -- (0,-2+.25);\draw[->] (0,-2-.25) -- (0,-3+.25);\draw[->] (0,-3-.25) -- (0,-4+.25);\draw[->] (0,-4-.25) -- (0,-5+.25);
\draw (0,0) node {124 ($q^{7}t^{2}$)};\draw (0,-1) node {240 ($q^{6}t^{3}$)};\draw (0,-2) node {113 ($q^{5}t^{4}$)};\draw (0,-3) node {130 ($q^{4}t^{5}$)};\draw (0,-4) node {102 ($q^{3}t^{6}$)};\draw (0,-5) node {020 ($q^{2}t^{7}$)};
\end{scope}
\begin{scope}[shift={(9,-4.5)}, >=latex]
\draw[->] (0,0-.25) -- (0,-1+.25);\draw[->] (0,-1-.25) -- (0,-2+.25);\draw[->] (0,-2-.25) -- (0,-3+.25);
\draw (0,0) node {123 ($q^{6}t^{3}$)};\draw (0,-1) node {230 ($q^{5}t^{4}$)};\draw (0,-2) node {202 ($q^{4}t^{5}$)};\draw (0,-3) node {021($q^{3}t^{6}$)};
\end{scope}
\begin{scope}[shift={(12,-3)}, >=latex]
\draw[->] (0,0-.25) -- (0,-1+.25);\draw[->] (0,-1-.25) -- (0,-2+.25);\draw[->] (0,-2-.25) -- (0,-3+.25);\draw[->] (0,-3-.25) -- (0,-4+.25);\draw[->] (0,-4-.25) -- (0,-5+.25);\draw[->] (0,-5-.25) -- (0,-6+.25);\draw[->] (0,-6-.25) -- (0,-7+.25);
\draw (0,0) node {224 ($q^{8}t^{2}$)};\draw (0,-1) node {241 ($q^{7}t^{3}$)};\draw (0,-2) node {213 ($q^{6}t^{4}$)};\draw (0,-3) node {131 ($q^{5}t^{5}$)};\draw (0,-4) node {112 ($q^{4}t^{6}$)};\draw (0,-5) node {120 ($q^{3}t^{7}$)};\draw (0,-6) node {200 ($q^{2}t^{8}$)};\draw (0,-7) node {001 ($q^{1}t^{9}$)};
\end{scope}
\begin{scope}[shift={(15,-2)}, >=latex]
\draw[->] (0,0-.25) -- (0,-1+.25);\draw[->] (0,-1-.25) -- (0,-2+.25);\draw[->] (0,-2-.25) -- (0,-3+.25);\draw[->] (0,-3-.25) -- (0,-4+.25);\draw[->] (0,-4-.25) -- (0,-5+.25);\draw[->] (0,-5-.25) -- (0,-6+.25);\draw[->] (0,-6-.25) -- (0,-7+.25);
\draw (0,0) node {135 ($q^{9}t^{1}$)};\draw (0,-1) node {134 ($q^{8}t^{2}$)};\draw (0,-2) node {133 ($q^{7}t^{3}$)};\draw (0,-3) node {132 ($q^{6}t^{4}$)};\draw (0,-4) node {122 ($q^{5}t^{5}$)};\draw (0,-5) node {220 ($q^{4}t^{6}$)};\draw (0,-6) node {201 ($q^{3}t^{7}$)};\draw (0,-7) node {011 ($q^{2}t^{8}$)};
\end{scope}
\end{tikzpicture}
$$}
\caption{The $f$-chains of nonzero lengths in Example \ref{ex:chain4-2}.}
\label{fig:f chain 4mx4}
\end{figure}

\noindent  

Moreover, $D_1=\{246, 235, 135\}$, $D_2=\{224,124,123,024,023,022\}$, $D_3=\{012,002,013\}$, $D'_1=\{010,020,021\}$, $D'_2=\{000, 001, 011,022,012,002\}$, $D'_3=\{024,023,013\}$.
The initial points of chains have generating function
\[ C_I(q,t)=
 q^{12}t^{0} +q^{10}t^{1} +q^{7}t^{2} +q^{6}t^{3} +q^{8}t^{2} +q^{9}t^{1}
+q^{6}t^{2} +q^{5}t^{3} +q^{4}t^{4} +q^{3}t^{5} +q^{2}t^{6} +q^{4}t^{4}. \]
The terminal points of chains have generating function (see Figure \ref{fig:C4tri} Right)
\[ C_T(q,t)=
 q^{0}t^{12} +q^{1}t^{10} +q^{2}t^{7} +q^{3}t^{6} +q^{1}t^{9} +q^{2}t^{8}
+q^{6}t^{2} +q^{5}t^{3} +q^{4}t^{4} +q^{3}t^{5} +q^{2}t^{6} +q^{4}t^{4}. \] 
We see that $C_T(q,t)=C_I(t,q)$, so that $C_W(q,t)=C_W(t,q)$. 

\end{example}

\begin{thm}\label{thm:explicit-coeffs}
For $j,k\in\N$, let $c(j,k)$ be the coefficient of $q^j t^k$ in 
$C^{(m)}_4(q,t)$. For $a\in\R$, define $\lfloor a\rfloor^+=\max(\lfloor a\rfloor,0)$. 

\noindent {\rm(a)} If $j+k>4m$, then  
\small$$
c(j,k)=\min \left(\Big\lfloor\frac{-6m+2+3j+k}{2}\Big\rfloor^+,\;
                 \Big\lfloor\frac{6m+2-j-k}{2}\Big\rfloor^+,\;
                 \Big\lfloor\frac{-6m+2+j+3k}{2}\Big\rfloor^+\right).
$$\normalsize\\
{\rm(b)} If $j+k=4m$, then 
\small$$ 
c(j,k)=\min \left(\Big\lfloor\frac{-m+2+j}{2}\Big\rfloor^+,\;
                 %\Big\lfloor\frac{3m+2-j}{2}\Big\rfloor^+\right).
		% NL: Double-check edit.
                 \Big\lfloor\frac{-m+2+k}{2}\Big\rfloor^+\right).
$$\normalsize\\
{\rm(c)} If $j+k<4m$, then $c(j,k)=0$.

\noindent As a consequence, the sequence $(c(d,0),c(d-1,1),c(d-2,2),\ldots,c(1,d-1),c(0,d))$
is unimodal for every positive integer $d$.
\end{thm}
\begin{proof} We may assume $j\ge k$ because of the $q,t$-joint symmetry. 
Using Remark \ref{rem:coefficient},
\begin{equation}\label{eq:chi}
c(j,k)=\Big|C_I(q,t)\Big|^{\deg=j+k}_{\qdeg\ge j}-\Big|C_I(q,t)\Big|^{\deg=j+k}_{\tdeg> j}.
\end{equation}

First, observe that all monomials in $C_I(q,t)$ have degree at least $4m$, so $c(j,k)=0$ if $j+k<4m$. This proves (c).

Then we consider the case $j+k>4m$. Because of \eqref{eq:sum}, $|C_I(q,t)|^{\deg=j+k}_{\qdeg\ge j}$ is equal to the cardinality of the set
$$\{(u,v):\; 0\le u\le m,\;  u\le v\le 2m-u,\;  6m-u-v=j+k,\;  j\le 6m-u-2v\}.$$
In this set, $\max(6m-j-2k,0)\le u\le (6m-j-k)/2$, and $v$ is determined by $u$.
Thus the cardinality is 
$\max((6m-j-k)/2-\max(6m-j-2k,0)+1,0)$, that is,
$$\min\left(\Big\lfloor \frac{-6m+2+j+3k}{2}\Big\rfloor^+, 
\Big\lfloor\frac{6m+2-j-k}{2}\Big\rfloor^+\right).$$
On the other hand, $|C_I(q,t)|^{\deg=j+k}_{\tdeg> j}$ is the cardinality of the set 
$$\{(u,v):0\le u\le m,\; u\le v\le 2m-u,\; 6m-u-v=j+k,\; j<v\}.$$
The conditions imply $0\le u\le 6m-2j-k$, but $6m-2j-k<2m-j<0$, so the set is empty. This implies (a).

Finally we consider the case $j+k=4m$. The degree-$4m$ part of $C_I(q,t)$ is equal to 
\begin{equation}\label{eq:j+k=4m}
\sum_{u=0}^{m-1}\sum_{v=m}^{u+m}q^{u+v}t^{4m-u-v}+\sum_{u=0}^mq^{2m+u}t^{2m-u}=\sum_{u=0}^{m}\sum_{v=m}^{u+m}q^{u+v}t^{4m-u-v}
\end{equation}
which is $q,t$-symmetric thanks to the bijection $(u,v)\mapsto (2m-v,2m-u)$. Therefore all $f$-chains  
have length $0$, and $c(j,k)$ is equal to the coefficient of $q^jt^k$ in \eqref{eq:j+k=4m}. This implies (b). 
\end{proof}

\section{Comparison to the Garsia-Haiman Formula}
\label{sec:proof-GH}

This section proves Conjecture~\ref{conj:qtcat} for all $m\geq 1$
and all $n\leq 4$. Fix $m\geq 1$. For $n=1$, we have
\[ AC_1^{(m)}(q,t)=\frac{(1-q)(1-t)}{(1-t)(1-q)}=1=C_1^{(m)}(q,t). \]
For $n=2$, we compute:
\begin{eqnarray*}
AC_2^{(m)}(q,t)&=&
 \frac{(q^1t^0)^{m+1}(1-q)(1-t)(1+q)(1-q)}
      {(q-t)(1-t)(1-q^2)(1-q)}
\\& & +\frac{(q^0t^1)^{m+1}(1-q)(1-t)(1+t)(1-t)}
      {(1-t^2)(1-t)(t-q)(1-q)}
\\ &=& \frac{q^m}{1-t/q}+\frac{t^m}{1-q/t}=\frac{q^{m+1}-t^{m+1}}{q-t}
\\ &=& q^m+q^{m-1}t+q^{m-2}t^2+\cdots+t^m=C_2^{(m)}(q,t).
\end{eqnarray*}

\subsection{The Case $n=3$.}
\label{subsec:GHn=3}

The partitions of $n=3$ are $(3)$, $(1,1,1)$ and $(2,1)$.
Using the Garsia-Haiman formula~\eqref{eq:GH-formula} gives
\begin{eqnarray*}
AC_3^{(m)}(q,t) &=&
\frac{(q^3t^0)^{m+1}(1-q)(1-t)(1+q+q^2)(1-q)(1-q^2)}
     {(q^2-t)(q-t)(1-t)(1-q^3)(1-q^2)(1-q)}
\\& & +\frac{(q^0t^3)^{m+1}(1-q)(1-t)(1+t+t^2)(1-t)(1-t^2)}
     {(1-t^3)(1-t^2)(1-t)(t^2-q)(t-q)(1-q)}
\\& & +\frac{(q^1t^1)^{m+1}(1-q)(1-t)(1+q+t)(1-q)(1-t)}
     {(q-t^2)(1-t)(1-t)(t-q^2)(1-q)(1-q)}.
\end{eqnarray*}
By cancelling common factors and using
the map $\sigma$ on $\Q(q,t)$ that sends $F(q,t)$ to $F(q,t)+F(t,q)$,
we can rewrite this as
\[ AC_3^{(m)}(q,t)=\sigma\left(\frac{q^{3m}}{(1-t/q^2)(1-t/q)}\right)
  +\frac{q^mt^m(1+q+t)}{(1-t^2/q)(1-q^2/t)}. \]
On the other hand, recall from~\S\ref{subsec:prove-n=3} that for $n=3$,
$I=\{(0,i,i+m):0\leq i\leq m\}$ and $C_I(q,t)=\sum_{v=0}^m q^{3m-2v}t^v$.
As $3m-2v\geq v$ for $0\leq v\leq m$, we can apply Lemma~\ref{lem:CI-to-CW}(b).
Before doing so, we rewrite $C_I(q,t)$ as follows:
\begin{eqnarray*}
C_I(q,t) &=& \sum_{v=0}^m q^{3m-2v}t^v =q^{3m}\sum_{v=0}^m (t/q^2)^v
 \\ &=& q^{3m}\left(\frac{1-(t/q^2)^{m+1}}{1-t/q^2}\right)
  =\frac{q^{3m}}{1-t/q^2}+\frac{q^mt^m}{1-q^2/t}.
\end{eqnarray*}
By the lemma,
\begin{eqnarray*}
&C_W&(q,t)\\
&=&\sigma\left(\frac{C_I(q,t)}{1-t/q}\right)
\\&=&\sigma\left(\frac{q^{3m}}{(1-t/q^2)(1-t/q)}\right)
+q^mt^m\bigg{[}\frac{1}{(1-q^2/t)(1-t/q)}+\\
& &\frac{1}{(1-t^2/q)(1-q/t)}\bigg{]}
\\&=&\sigma\left(\frac{q^{3m}}{(1-t/q^2)(1-t/q)}\right)
 +\frac{q^mt^m(1+q+t)}{(1-t^2/q)(1-q^2/t)}, 
\end{eqnarray*}
where the last equality follows by routine algebra.
We now see that $C_3^{(m)}(q,t)=AC_3^{(m)}(q,t)$.

\subsection{The Case $n=4$.}
\label{subsec:GHn=4}

The partitions of $n=4$ are $(4)$, $(1,1,1,1)$, $(3,1)$, $(2,1,1)$
and $(2,2)$. Writing out the Garsia-Haiman formula~\eqref{eq:GH-formula}
and cancelling common factors gives
\begin{eqnarray*}
 AC_4^{(m)}(q,t) &=&
 \frac{(q^6t^0)^{m+1}}{(q^3-t)(q^2-t)(q-t)}
+\frac{(q^0t^6)^{m+1}}{(t^3-q)(t^2-q)(t-q)}
\\&&+\frac{(q^3t^1)^{m+1}(1+q+q^2+t)}{(q^2-t^2)(q-t)(t-q^3)}
+\frac{(q^1t^3)^{m+1}(1+t+t^2+q)}{(q-t^3)(t^2-q^2)(t-q)}
\\&& +\frac{(q^2t^2)^{m+1}(1-qt)}{(q-t^2)(q-t)(t-q^2)(t-q)}. 
\end{eqnarray*}
We can rewrite this as 
\begin{eqnarray*}
 AC_4^{(m)}(q,t)&=&
\sigma\left(\frac{q^{6m}}{(1-t/q)(1-t/q^2)(1-t/q^3)}\right)
\\&&-\sigma\left(q^{3m}t^m\frac{t/q+t/q^2+t/q^3+t^2/q^3}
  {(1-t^2/q^2)(1-t/q)(1-t/q^3)}\right)
\\&&+q^{2m}t^{2m}\frac{q^2t^2(1-qt)}{(q-t^2)(t-q^2)(q-t)(t-q)}.
\end{eqnarray*}

Recall from Lemma~\ref{lem:W-B union} that $I=I_4^{(m)}$
is the disjoint union $I=D_1\cup D_2\cup D_3$. 
We derived formulas for $C_{D_3}(q,t)$ and $C_{D_1\cup D_2}(q,t)$
in~\eqref{eq:sum'} and~\eqref{eq:sum}. Adding these formulas gives
\begin{equation}\label{eq:newCIsum}
\aligned
C_I(q,t)
&=\sum_{u=0}^{m-1}\sum_{v=m}^{u+m}q^{u+v}t^{4m-u-v}
 +\sum_{u=0}^m\sum_{v=u}^{2m-u}q^{6m-u-2v}t^v\\
&=\sum_{u=0}^{m-1}\sum_{v=m}^{u+m}q^{u+v}t^{4m-u-v}
 +\sum_{u=0}^mq^{2m+u}t^{2m-u}
 +\sum_{u=0}^{m-1}\sum_{v=u}^{2m-u-1}q^{6m-u-2v}t^v\\
&=\sum_{u=0}^{m}\sum_{v=m}^{u+m}q^{u+v}t^{4m-u-v}
 +\sum_{u=0}^{m-1}\sum_{v=u}^{2m-u-1}q^{6m-u-2v}t^v.\\
\endaligned
\end{equation}
To continue, write $I=I'\cup I''$, where 
$I'=\{ w\in I: \area(w)+\dinv_m(w)=4m \}$ and $I''=I\setminus I'$.
Note that the first double sum in~\eqref{eq:newCIsum} is $C_{I'}(q,t)$,
and the second double summation is $C_{I''}(q,t)$.
Also write $T=T'\cup T''$ and $W=W'\cup W''$,
where $T'$ (resp. $W'$) consists of the objects $w$ in $T$ (resp. $W$)
with $\area(w)+\dinv_m(w)=4m$. The map $f$ restricts to
decompose $W'$ (resp. $W''$) into $f$-chains that go
from $I'$ to $T'$ (resp. $I''$ to $T''$), and we have
$C_{T'}(q,t)=C_{I'}(t,q)$ and $C_{T''}(q,t)=C_{I''}(t,q)$.

It is routine to check that $C_{I'}(q,t)=C_{I'}(t,q)$.
Hence, Lemma~\ref{lem:CI-to-CW}(a) applies to give 
\begin{equation}\label{eq:CI'}
\aligned
&C_{W'}(q,t)=C_{I'}(q,t)=\sum_{u=0}^m\sum_{v=m}^{u+m} q^{u+v}t^{4m-u-v}\\
&% \sum_{u=0}^{m}\sum_{v=m}^{u+m}q^{u+v}t^{4m-u-v}
=\sum_{u=0}^{m}q^ut^{4m-u}\sum_{v=m}^{u+m}(q/t)^v=\sum_{u=0}^{m}q^ut^{4m-u}\frac{(q/t)^m-(q/t)^{m+u+1}}{1-q/t}\\
&=\frac{q^mt^{3m}}{1-q/t}\sum_{u=0}^{m}(q/t)^u-\frac{q^{m+1}t^{3m-1}}{1-q/t}\sum_{u=0}^{m}(q^2/t^2)^u\\
&=\frac{q^mt^{3m}}{(1-q/t)^2}-\frac{q^{2m}t^{2m}\cdot q/t}{(1-q/t)^2}-\frac{q^{m}t^{3m}\cdot q/t}{(1-q/t)(1-q^2/t^2)}+\frac{q^{3m}t^{m}\cdot q^3/t^3}{(1-q/t)(1-q^2/t^2)}\\
&=\frac{q^mt^{3m}}{(1-q/t)(1-q^2/t^2)}+\frac{q^{3m}t^m}{(1-t/q)(1-t^2/q^2)}
-\frac{q^{2m}t^{2m}\cdot q/t}{(1-q/t)^2}\\
&=\sigma\Big(\frac{q^{3m}t^{m}}{(1-t/q)(1-t^2/q^2)}\Big)
-\frac{q^{2m}t^{2m}\cdot q/t}{(1-q/t)^2}.\\
\endaligned
\end{equation}

On the other hand, since $6m-u-2v\geq v$ for all $u,v$ appearing
in the second sum in~\eqref{eq:newCIsum}, we can compute
$C_{W''}(q,t)$ using Lemma~\ref{lem:CI-to-CW}(b). First we calculate 
$C_{I''}(q,t)$ to be
\begin{equation}\label{eq:CI''}
\aligned
&\sum_{u=0}^{m-1}\sum_{v=u}^{2m-u-1}q^{6m-u-2v}t^v
=\sum_{u=0}^{m-1}q^{6m-u}\sum_{v=u}^{2m-u-1}(t/q^2)^v\\
&=\sum_{u=0}^{m-1}q^{6m-u}\frac{(t/q^2)^u-(t/q^2)^{2m-u}}{1-t/q^2}\\
&=\frac{q^{6m}}{1-t/q^2}\sum_{u=0}^{m-1}(t/q^3)^u
-\frac{q^{2m}t^{2m}}{1-t/q^2}\sum_{u=0}^{m-1}(q/t)^u\\
&=\frac{q^{6m}(1-(t/q^3)^{m})}{(1-t/q^2)(1-t/q^3)}
-\frac{q^{2m}t^{2m}(1-(q/t)^{m})}{(1-t/q^2)(1-q/t)}\\
&=\frac{q^{6m}}{(1-t/q^2)(1-t/q^3)}-\frac{q^{2m}t^{2m}}{(1-t/q^2)(1-q/t)}
+\frac{q^{3m}t^m(q/t-t/q^3)}{(1-t/q^2)(1-t/q^3)(1-q/t)}\\
&=\frac{q^{6m}}{(1-t/q^2)(1-t/q^3)}-\frac{q^{2m}t^{2m}}{(1-t/q^2)(1-q/t)}
+\frac{q^{3m}t^m(q/t+1/q)}{(1-t/q^3)(1-q/t)}.\\
\endaligned
\end{equation}
Dividing by $(1-t/q)$ and applying $\sigma$, the first term here becomes
$$\sigma\left(\frac{q^{6m}}{(1-t/q)(1-t/q^2)(1-t/q^3)}\right).$$ 
The second term here, when combined with 
the second term in~\eqref{eq:CI'}, becomes
% checked twice in Mathematica.
\begin{multline*}
 q^{2m}t^{2m}\left[
\frac{-q/t}{(1-q/t)^2}-\sigma\Big(\frac{1}{(1-t/q)(1-t/q^2)(1-q/t)}\Big)\right]
\\ =q^{2m}t^{2m}\frac{q^2t^2(1-qt)}{(q-t^2)(t-q^2)(q-t)(t-q)}.
\end{multline*}
Finally, the third term in~\eqref{eq:CI''},
when combined with the first term in~\eqref{eq:CI'}, becomes
% checked twice in Mathematica.
\begin{multline*}
 \sigma\left(q^{3m}t^m\left[ \frac{1}
{(1-t/q)(1-t^2/q^2)}+\frac{q/t+1/q}{(1-t/q)(1-t/q^3)(1-q/t)} \right]\right)
\\ =-\sigma\left(q^{3m}t^m\frac{t/q+t/q^2+t/q^3+t^2/q^3}
{(1-t^2/q^2)(1-t/q)(1-t/q^3)}\right).
\end{multline*}
Adding up all the pieces, we get $C_4^{(m)}(q,t)=C_W(q,t)=AC_4^{(m)}(q,t)$.

\section{Rational-Slope $q,t$-Catalan Polynomials}
\label{sec:rat-slope}

As mentioned in the Introduction, the combinatorial formulas for
higher $q,t$-Catalan polynomials can be interpreted as generating
functions for ``$m$-Dyck paths,'' which are lattice paths contained
within the triangle with vertices $(0,0)$, $(mn,n)$, and $(0,n)$.
More generally, one can also define versions of the $q,t$-Catalan polynomials
counting lattice paths staying within other triangles.  
We shall focus on proving joint symmetry of rational-slope
$q,t$-Catalan polynomials for triangles of height 4. At the end of 
the section, we briefly discuss Gorsky and Mazin's
proof of joint symmetry~\cite{GM2} for triangles of height 3.

We begin by reviewing some definitions and results from~\cite{LW}.
For $r,s\in\N^+$, define an \emph{$r\times s$ Dyck path}
to be a lattice path from $(0,0)$ to $(r,s)$ that lies above
the diagonal line segment joining $(0,0)$ to $(r,s)$.
Let $L_{r,s}^+$ be the set of $r\times s$ Dyck paths. 
For a path $\pi\in L_{r,s}^+$, let $\area(\pi)$ be the number 
of lattice squares that lie entirely above the diagonal and below $\pi$. 
The lattice squares above $\pi$ in the triangle with vertices
$(0,0)$, $(0,s)$ and $(r,s)$ form a partition diagram denoted $D(\pi)$.
%For each square $c\in D(\pi)$, let the \emph{arm of $c$} (denoted $a(c)$)
%be the number of squares in $D(\pi)$ to the right of $c$, and
%let the \emph{leg of $c$} (denoted $l(c)$)
%be the number of squares in $D(\pi)$ below $c$. Define
Define
\begin{eqnarray*}
h^+_{r/s}(\pi)&=&\sum_{c\in D(\pi)}
\chi\Big(\frac{a(c)}{l(c)+1}\le r/s< \frac{a(c)+1}{l(c)}\Big); \\
h^-_{r/s}(\pi)&=&\sum_{c\in D(\pi)}
\chi\Big(\frac{a(c)}{l(c)+1}< r/s\le \frac{a(c)+1}{l(c)}\Big). 
\end{eqnarray*}
The \emph{rational $q,t$-Catalan number for $r\times s$ Dyck paths}
is defined as
$$C_{r',s',n'}(q,t)=\sum_{\pi\in L_{r,s}^+}q^{\area(\pi)}t^{h^+_{r/s}(\pi)}$$
where $n'=\gcd(r,s)$, $r'=r/n'$, and $s'=s/n'$.  (This indexing convention
is used to match the notation in~\cite{LW}.) 
In~\cite{LW}, an involution $I:L_{r,s}^+\rightarrow L_{r,s}^+$ was
defined that preserves area and interchanges $h^+_{r/s}$ and $h^-_{r/s}$.
It follows that we could replace $h^+_{r/s}$ by $h^-_{r/s}$ in the
definition of $C_{r',s',n'}(q,t)$.

The rational $q,t$-Catalan number $C_{nm+1,n,1}(q,t)$
is exactly the same as $C_{m,1,n}(q,t)$ because the natural 
inclusion $\iota$ of $L_{nm,n}^+$ into  $L_{nm+1,n}^+$ is bijective, 
and $h^+_m(\pi)=h^+_{(mn+1)/n}(\iota(\pi))$ for $\pi\in L_{nm,n}^+$. 
Furthermore, it was shown in~\cite[Lemma 6.3.3]{HHLRU} that 
$h_m^+(\pi)=\dinv_m(\pi)$, which implies $C_{m,1,n}(q,t)=C^{(m)}_n(q,t)$.  
Thus we only need to discuss the two cases of
$(4m+2)\times 4$ and $(4m-1)\times 4$ Dyck paths.
In the following two subsections, we identify $r\times 4$ Dyck paths $\pi$
with ``$r\times 4$ Dyck words'' $(\gamma_0,\gamma_1,\gamma_2,\gamma_3)$
as follows: number the rows of the triangle zero to three, from bottom to top.
Then for $i=0,1,2,3$, the number of cells in the $i$-th row that lie 
above $\pi$ is $mi-\gamma_i$.

\subsection{$(4m+2) \times 4$ Dyck paths} 
\label{subsec:4m+2 by 4}

Assume $m\ge 1$. The analogue of $W_n^{(m)}$ in \S3.1 in the $(4m+2)\times 4$ case is
$$\W=\{\gamma=(0,\gamma_1,\gamma_2,\gamma_3):  \gamma_1\ge0,\gamma_2\ge-1,\gamma_3\ge-1, \gamma_{i+1}\le\gamma_i+m \textrm{ for } i=0,1,2\}. $$

\begin{lem}\label{lem:4m and 4m+2}
 For any $\pi\in L^+_{4m+2,4}$, we have $h^+_m(\pi)=h^-_{(4m+2)/4}(\pi)$.
\end{lem}
\begin{proof} For each $\pi$, it suffices to show that for any cell 
$c\in D(\pi)$, the two conditions $a(c)/(l(c)+1)\leq m< (a(c)+1)/l(c)$ 
and $a(c)/(l(c)+1)<(4m+2)/4\le(a(c)+1)/l(c)$,
which appeared in the definition of $h^+_m(\pi)$ and $h^-_{(4m+2)/4}(\pi)$, are equivalent. Indeed, since $0\le l(c)\le 2$, the fraction $a(c)/(l(c)+1)$ is in the open interval $(m,(4m+2)/4)$ only if $a(c)=3m+1$ and $l(c)=2$; but the latter is impossible since $a(c)\le 3m$. Thus $a(c)/(l(c)+1)\notin(m,(4m+2)/4)$. Similarly, $(a(c)+1)/l(c)\notin(m,(4m+2)/4)$. 
\end{proof}

\begin{prop} $C_{2m+1,2,2}(q,t)=C_{2m+1,2,2}(t,q)$.
%$\sum_{\pi\in L_{4m+2,4}^{+}} q^{\emph{area}(\pi)}t^{h^{+}_{(4m+2)/4}(\pi)}=\sum_{\pi\in L_{4m+2,4}^{+}} t^{\emph{area}(\pi)}q^{h^{+}_{(4m+2)/4}(\pi)}.$
\end{prop}

\begin{proof}
We can replace $h^{+}_{(4m+2)/4}$ by $h^{-}_{(4m+2)/4}$ using the involution at $(4m+2)/4$ defined in \cite{LW} and  then replace the latter by  $h^{+}_m$, thanks to Lemma \ref{lem:4m and 4m+2}. Thus we need to prove
\begin{equation}\label{eq:4m+2 to 4m}
\sum_{\pi\in L_{4m+2,4}^{+}} q^{\area(\pi)}t^{h^{+}_{m}(\pi)}=\sum_{\pi\in L_{4m+2,4}^{+}} t^{\area(\pi)}q^{h^{+}_{m}(\pi)}.
\end{equation}

The proof of (\ref{eq:4m+2 to 4m}) is similar to that for the $4m\times 4$ case (\S4.3).
The bijection $f_0$ is defined the same way as Definition \ref{defn:f0} except that the domain $A_0$ consists of $\gamma\in \W$ that satisfy
the following condition:  let $q\ge2$ be the smallest integer such that $\gamma_q-\gamma_{q-2}\le m$, or let $q=n(=4)$ if there is no such integer, 
then  $\gamma_{q-1}-1\le \gamma_{n-1}+m$ and $\gamma_{q-1}\ge0$; moreover, $\gamma_2\neq -1$ (otherwise $f(\gamma)=(0,-1,*,*)$ is not in $\W$). A case-by-case study of $q=2,3,4$ shows
$$\W\setminus A_0=\{(0,\gamma_1,-1,\gamma_3)\in\W\}\cup\{\gamma\in\W: \gamma_2>m,\; \gamma_3<\gamma_2-m-1\},$$
and $B_0=f_0(A_0)$ consists of $\gamma\in\W$ such that one of the following holds:

(a)  $\gamma_3\le m-1$, or

(b) $\gamma_3>m-1$, $\gamma_3\le \gamma_1+m-1$, $\gamma_2\le \gamma_3+m+1$, or

(c) $\gamma_2>m$, $\gamma_3-\gamma_1>m-1$, $\gamma_3\le \gamma_2+m-1$.

\noindent Further computation shows
\small$$
\W\setminus B_0=\{\gamma\in\W: \gamma_1+m\le\gamma_3\le\gamma_2+m-1, \gamma_2\le m\}\cup\{\gamma\in\W: \gamma_1\le\gamma_2, \gamma_3=\gamma_2+m\}.
$$\normalsize

Next, $A_1, B_1, f_1, A, B,$ and $f$ are defined by the same formula as Definition \ref{def:f1} (but the meaning of $\W$ is different). Lemma \ref{lem:gamma-A} should be revised to say:
$$\W\setminus A=\{(0,\gamma_1,-1,\gamma_3)\in\W\}.$$
The map $f:A\to B$ is also a bijection and changes $(\area,h^+_m)$ to $(\area-1,h^+_m+1)$.
Lemma \ref{lem:W-B union} should be revised to say:

{\rm(1)} $B_1=\{\gamma\in\W: \gamma_2\le m-1, \gamma_3-\gamma_1\ge m\}$.

{\rm(2)} $I=\W\setminus B$ is equal to the disjoint union $D_1\cup D_2$, where
$$\aligned
&D_1=\{(0,\gamma_1,\gamma_2,\gamma_2+m): 0\le\gamma_1\le m\le \gamma_2\le\gamma_1+m\},\\
&D_2=\{(0,\gamma_1,m,\gamma_3):0\le\gamma_1\le m-1, \gamma_1+m\le\gamma_3\le 2m-1\}.
\endaligned$$

We claim that the analogue of \eqref{eq:head-tail symmetry} still holds.
Note that $\W\setminus A=D'_1\cup D'_2$, where $D'_1=\{(0,\gamma_1,-1,\gamma_3)\in\W : \gamma_1> \gamma_3\}$, and 
$D'_2=\{(0,\gamma_1,-1,\gamma_3)\in\W : \gamma_1\le \gamma_3\}$. 
Moreover,
$$
(\area(\gamma),h^+_m(\gamma))=
\begin{cases}
(\gamma_1+\gamma_3+1,6m+1-3\gamma_1-\gamma_3),& \textrm{ if } \gamma\in D'_1;\\
(\gamma_1+\gamma_3+1,6m-\gamma_1-3\gamma_3), & \textrm{ if } \gamma\in D'_2;\\
(\gamma_1+2\gamma_2+2, 2m-\gamma_2),& \textrm{ if } \gamma\in D_1;\\
(\gamma_1+\gamma_3+m+2,3m-\gamma_3), & \textrm{ if } \gamma\in D_2.\\
\end{cases}
$$
The analogue of (\ref{eq:head-tail symmetry-12D}) holds where $\Delta$ is the set of lattice points that are either inside or on the boundary of the triangle with vertices $(0,6m+2)$, $(m,3m+2)$, $(2m,2m+2)$.
Indeed, the $(\area,h^+_m)$ pairs for $\gamma$ in ${D'_1}$ (resp. $D'_2$) form the set $\{(x,y)\in \Delta: x+y\textrm { is even}\}$ (resp.~$\{(x,y)\in \Delta: x+y\textrm { is odd}\}$), and the $(h^+_m,\area)$ pairs for $\gamma$ in $D_1$ (resp.~$D_2$) form the set $\{(x,y)\in \Delta: x\le m\}$ (resp.~$\{(x,y)\in \Delta: x>m\}$).
\end{proof}

We can express $C_{2m+1,2,2}$ in the form similar to Garsia-Haiman formula as
\small
$$
\sigma\Big(
q^{6 m}\frac{q^8}{(q - t)(q^2 - t)(q^3 - t)} - 
 q^{3 m} t^m\frac{q^4t(1 + q)}{(q - t)^2(q^3 - t)} 
\Big)
+ 
 q^{2 m} t^{2 m}\frac{q^2 t^2(q^2t+qt^2-q^2-t^2)}{(q - t)^2(q^2 - t)(t^2-q)}.
$$
\normalsize
Moreover, we have the following analogue of Theorem \ref{thm:explicit-coeffs} and the proof is omitted.
\begin{thm}
For $j,k\in\N$, let $c(j,k)$ be the coefficient of $q^j t^k$ in 
$C_{2m+1,2,2}(q,t)$. For $a\in\R$, define $\lfloor a\rfloor^+=\max(\lfloor a\rfloor,0)$. 
Then  for $4m+2\le j+k\le 6m+2$,
\small$$
c(j,k)=\min \left(\Big\lfloor\frac{-6m+3j+k}{2}\Big\rfloor^+,\;
                 \Big\lfloor\frac{6m+4-j-k}{2}\Big\rfloor^+,\;
                 \Big\lfloor\frac{-6m+j+3k}{2}\Big\rfloor^+\right).
$$
\normalsize
\noindent Otherwise $c(j,k)=0$. As a consequence, the sequence $(c(d,0),c(d-1,1),c(d-2,2),\ldots,c(1,d-1),c(0,d))$
is unimodal for every positive integer $d$.
\end{thm}

\subsection{$(4m-1) \times 4$ Dyck paths}\label{4m-1} 
\label{subsec:4m-1 by 4}

Assume $m\ge 1$.  The analogue of $W_n^{(m)}$ in \S3.1 in the $(4m-1)\times 4$ case is
$$\W=\{\gamma=(0,\gamma_1,\gamma_2,\gamma_3):  \gamma_1,\gamma_2,\gamma_3\ge1,\gamma_{i+1}\le\gamma_i+m \textrm{ for } i=0,1,2\}. $$

\begin{lem}\label{lem:4m and 4m-1}
 For any $\pi\in L^+_{4m-1,4}$, $h^-_m(\pi)=h^+_{(4m-1)/4}(\pi)$.
\end{lem}

\begin{proof} Similar to the proof of Lemma \ref{lem:4m and 4m+2}. Use the fact that there is no fraction in the open interval $((4m-1)/4,m)$ with denominator at most $3$.
\end{proof}

We need the following property of the involution $I_{{4m}/{4}}$ defined in \cite{LW}: the restriction of $I_{{4m}/{4}}$ to $L^+_{4m-1,4}$, denoted by $I$, is an involution since it does not change the number of arrows heading to the vertex 0 in the multigraph. This involution $I$ exchanges $h^+_m$ with $h^-_m$ and keeps $\area$ unchanged.

\begin{prop} \label{prop:4m-1,4}
$C_{4m-1,4,1}(q,t)=C_{4m-1,4,1}(t,q)$.
%\begin{equation}\label{eq:4m-1,4}
%\sum_{\pi\in L_{4m-1,4}^{+}} q^{\emph{area}(\pi)}t^{h^{+}_{(4m-1)/4}(\pi)}=\sum_{\pi\in L_{4m-1,4}^{+}} t^{\emph{area}(\pi)}q^{h^{+}_{(4m-1)/4}(\pi)}.
%\end{equation}
\end{prop}

\begin{proof}
Since $h^{+}_{(4m-1)/4}(\pi)=h^{-}_{m}(\pi)=h^{+}_m(I(\pi))$ and $\area(\pi)=\area(I(\pi))$, Proposition~\ref{prop:4m-1,4} is equivalent to the following after the substitution of $\pi$ by $I(\pi)$:
\begin{equation}\label{eq:4m-1 to 4m}
\sum_{\pi\in L_{4m-1,4}^{+}} q^{\area(\pi)}t^{h^{+}_m(\pi)}=\sum_{\pi\in L_{4m-1,4}^{+}} t^{\area(\pi)}q^{h^{+}_m(\pi)}.
\end{equation}
The proof of (\ref{eq:4m-1 to 4m}) is again similar to the $4m\times 4$ case (\S4.3). The bijection $f_0$ is defined the same way as in Definition \ref{defn:f0} except that the domain $A_0$ is the subset of $W$ consisting of $\gamma=(\gamma_0,\dots,\gamma_{n-1})$ (where $n=4$) that satisfies the following condition: let $q\ge2$ be the smallest integer such that $\gamma_q-\gamma_{q-2}\le m$, or let $q=n$ if there is no such integer, then  $\gamma_{q-1}-1\le \gamma_{n-1}+m$ and $\gamma_{q-1}>1$. Then
$$\W\setminus A_0=\{(0,1,\gamma_2,\gamma_3)\in\W: \gamma_2\le m\}\cup\{\gamma\in\W: \gamma_2-1>\gamma_3+m\}.$$
Next, $A_1, B_1, f_1, A, B,$ and $f$ are defined by the same formula as Definition \ref{def:f1} (although the meaning of $\W$ is different). Lemma \ref{lem:gamma-A} should be revised to say:
$$\W\setminus A=\{(0,1,\gamma_2,\gamma_3)\in\W: \gamma_2\le m\}.$$
The map $f:A\to B$ is a bijection and changes $(\area,h^+_m)$ to $(\area-1,h^+_m+1)$.
Lemma \ref{lem:W-B union} should be revised to say:

{\rm(1)} $B_1=\{\gamma\in\W: \gamma_1\ge2, \gamma_2\le m-1, \gamma_3-\gamma_1\ge m\}$.

{\rm(2)} $\W\setminus B$ is equal to the disjoint union $D_1\cup D_2\cup D_3$, where
$$
\aligned
&D_1=\{(0,\gamma_1,\gamma_2,\gamma_2+m): 2\le\gamma_1\le m\le \gamma_2\le\gamma_1+m\}\cup\{(0,1,m+1,2m+1)\},\\
&D_2=\{(0,\gamma_1,m,\gamma_3): \gamma_1\ge 2,\; \gamma_1+m\le\gamma_3\le 2m-1\},\\
&D_3=\{(0,1,\gamma_2,\gamma_3): 1\le\gamma_2\le m,\; 1+m\le \gamma_3\le\gamma_2+m\}.
\endaligned
$$

We claim that the analogue of \eqref{eq:head-tail symmetry} still holds. Indeed, $\W\setminus A=D'_1\cup D'_2\cup D'_3$, where 
$D'_1=\{(0,1,\gamma_2,\gamma_3)\in\W : \gamma_3<\gamma_2\le m\}$, 
$D'_2=\{(0,1,\gamma_2,\gamma_3)\in\W : \gamma_2\le \gamma_3\le m\}$, 
$D'_3=\{(0,1,\gamma_2,\gamma_3)\in\W :\gamma_2\le m, \gamma_3>m\}$. 
Moreover, 
$$
(\area(\gamma),h^+_m(\gamma))=
\begin{cases}
(\gamma_2+\gamma_3-2,\; 6m+2-3\gamma_2-\gamma_3), & \textrm{ if } \gamma\in D'_1;\\
(\gamma_2+\gamma_3-2,\; 6m+1-\gamma_2-3\gamma_3),& \textrm{ if } \gamma\in D'_2;\\
(\gamma_2+\gamma_3-2,\; 4m-\gamma_2-\gamma_3), & \textrm{ if } \gamma\in D'_3;\\
(\gamma_1+2\gamma_2+m-3,\;2m-\gamma_2),& \textrm{ if } \gamma\in D_1;\\
(\gamma_1+\gamma_3+m-3,\;3m-\gamma_3), & \textrm{ if } \gamma\in D_2;\\
(\gamma_2+\gamma_3-2,\;4m-\gamma_2-\gamma_3), & \textrm{ if } \gamma\in D_3.\\
\end{cases}
$$
The analogue of (\ref{eq:head-tail symmetry-12D}) holds where $\Delta$ is the set of lattice points that are either inside or on the boundary of the triangle with vertices $(0,6m-3)$, $(m-1,3m)$, $(2m-2,2m+1)$.
The symmetry follows from the observation that the $(\area,h^+_m)$ pairs for $\gamma$ in ${D'_1}$ (resp. $D'_2$) form the set $\{(x,y)\in \Delta: x+y\textrm { is even}\}$ (resp. $\{(x,y)\in \Delta: x+y\textrm { is odd}\}$), while the $(h^+_m,\area)$ pairs for $\gamma$ in $D_1$ (resp. $D_2$) form the set $\{(x,y)\in \Delta: x\le m\}$ (resp. $\{(x,y)\in \Delta: x>m\}$).

The analogue of (\ref{eq:head-tail symmetry-3}) can be proved by the one-to-one correspondence from $D'_3$ to $D_3$ that sends $(\gamma_2,\gamma_3)$ to $(2m+1-\gamma_3,2m+1-\gamma_2)$.
\end{proof}

\begin{example} \label{ex:chain4-2'}
Consider the $(4m-1)\times 4$ case for $m=2$. There are thirty objects
in $\W_{4m-1,4}$, which produce $f$-chains of nonzero lengths listed in Figure \ref{fig:f chain (4m-1)x4} as well as  
three chains of length zero: $124 (q^4t^2), 113 (q^2t^4), 123 (q^3t^3)$. 
These are subchains of the chains in Example~\ref{ex:chain4-2}. Note that $\gamma_1,\gamma_2,\gamma_3>0$.

\begin{figure}[t]
\iffalse
\epsfig{file=fig4.eps,scale=1.0}
\fi
{\tiny
$$
\begin{tikzpicture}[scale=.8]
\draw[dotted, thick] (-2,-6) -- (-.2,-6) (.2,-6) -- (2,-6) (4,-6) -- (5.8,-6) (6.2,-6) -- (8.8,-6) (9.2,-6)--(11,-6);
\begin{scope}[shift={(0,-1.5)},>=latex]
\draw[->] (0,0-.25) -- (0,-1+.25);\draw[->] (0,-1-.25) -- (0,-2+.25);\draw[->] (0,-2-.25) -- (0,-3+.25);\draw[->] (0,-3-.25) -- (0,-4+.25);\draw[->] (0,-4-.25) -- (0,-5+.25);\draw[->] (0,-5-.25) -- (0,-6+.25);\draw[->] (0,-6-.25) -- (0,-7+.25);\draw[->] (0,-7-.25) -- (0,-8+.25);\draw[->] (0,-8-.25) -- (0,-9+.25);
\draw (0,0) node {246 ($q^{9}t^0$)};\draw (0,-1) node {245 ($q^{8}t^1$)};\draw (0,-2) node {244 ($q^{7}t^2$)};\draw (0,-3) node {243 ($q^{6}t^3$)};\draw (0,-4) node {233 ($q^{5}t^4$)};\draw (0,-5) node {232 ($q^{4}t^5$)};\draw (0,-6) node {222 ($q^{3}t^6$)};\draw (0,-7) node {221 ($q^{2}t^7$)};\draw (0,-8) node {211 ($q^{1}t^8$)};\draw (0,-9) node {111 ($q^{0}t^9$)};
\end{scope}
\begin{scope}[shift={(3,-3)}, >=latex]
\draw[->] (0,0-.25) -- (0,-1+.25);\draw[->] (0,-1-.25) -- (0,-2+.25);\draw[->] (0,-2-.25) -- (0,-3+.25);\draw[->] (0,-3-.25) -- (0,-4+.25);\draw[->] (0,-4-.25) -- (0,-5+.25);\draw[->] (0,-5-.25) -- (0,-6+.25);
\draw (0,0) node {235 ($q^{7}t^{1}$)};\draw (0,-1) node {234 ($q^{6}t^{2}$)};\draw (0,-2) node {242 ($q^{5}t^{3}$)};\draw (0,-3) node {223 ($q^{4}t^{4}$)};\draw (0,-4) node {231 ($q^{3}t^{5}$)};\draw (0,-5) node {212 ($q^{2}t^{6}$)};\draw (0,-6) node {121 ($q^{1}t^{7}$)};
\end{scope}
\begin{scope}[shift={(6,-4.5)}, >=latex]
\draw[->] (0,0-.25) -- (0,-1+.25);\draw[->] (0,-1-.25) -- (0,-2+.25);\draw[->] (0,-2-.25) -- (0,-3+.25);\draw[->] (0,-3-.25) -- (0,-4+.25);
\draw (0,0) node {224 ($q^{5}t^{2}$)};\draw (0,-1) node {241 ($q^{4}t^{3}$)};\draw (0,-2) node {213 ($q^{3}t^{4}$)};\draw (0,-3) node {131 ($q^{2}t^{5}$)};\draw (0,-4) node {112 ($q^{1}t^{6}$)};
\end{scope}
\begin{scope}[shift={(9,-3.5)}, >=latex]	
\draw[->] (0,0-.25) -- (0,-1+.25);\draw[->] (0,-1-.25) -- (0,-2+.25);\draw[->] (0,-2-.25) -- (0,-3+.25);\draw[->] (0,-3-.25) -- (0,-4+.25);
\draw (0,0) node {135 ($q^{6}t^{1}$)};\draw (0,-1) node {134 ($q^{5}t^{2}$)};\draw (0,-2) node {133 ($q^{4}t^{3}$)};\draw (0,-3) node {132 ($q^{3}t^{4}$)};\draw (0,-4) node {122 ($q^{2}t^{5}$)};
\end{scope}
\end{tikzpicture}
$$}
\caption{The $f$-chains of nonzero lengths in Example \ref{ex:chain4-2'}.}
\label{fig:f chain (4m-1)x4}
\end{figure}
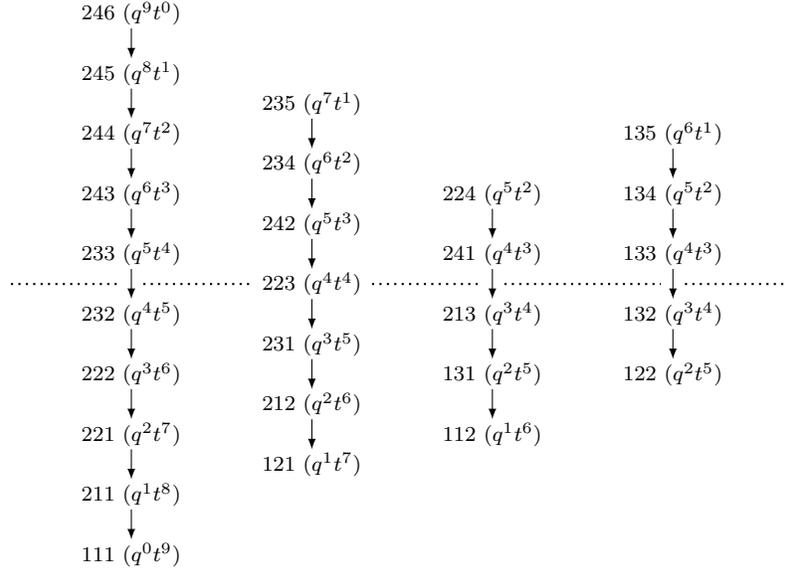

Moreover, $D_1=\{246, 235, 224, 135\}$, $D_2=\emptyset$, $D_3=\{124, 113, 123\}$, $D'_1=\{121\}$, $D'_2=\{111, 112, 122\}$, $D'_3=\{124, 113, 123\}$.
The set of initial points $I$ has generating function
$$ C_I(q,t)=q^9t^0 + q^7t^1 + q^5t^2 + q^6t^1+q^4t^2 + q^2t^4 + q^3t^3. $$
The set of terminal points $T$ has generating function
$$ C_T(q,t)=q^0t^9 + q^1t^7 + q^1t^6 + q^2t^5 + q^4t^2 + q^2t^4 + q^3t^3. $$
Since $C_T(q,t)=C_I(t,q)$, we get $C_W(q,t)=C_W(t,q)$.
\end{example}

\begin{rem}
For $n\le 3$, the only rational $q,t$-Catalan polynomials 
for $r\times n$ Dyck paths that are not equal to some $C^{(m)}_n(q,t)$ 
occur in the case when $r=3m-1$ and $n=3$. 
This case is similar to \S\ref{4m-1}: $W=\{(0,\gamma_1,\gamma_2): 1\le\gamma_1\le m,\; 1\le\gamma_2\le\gamma_1+m\}$, $W\setminus A_0=D_1'\cup D_2'$ where $D_1'=\{(0,1,\gamma_2): 1\le \gamma_2\le m\}$ and $D_2'=\{(0,1,m+1\}$;  $W\setminus B_0=D_1\cup D_2$ where $D_1=\{(0,\gamma_1,\gamma_1+m): 1\le \gamma_1\le m\}$ and $D_2=\{(0,1,m)\}$, and
$$
(\area(\gamma),h^+_m(\gamma))=
\begin{cases}
(\gamma_2-1,\; 3m-2\gamma_2), & \textrm{ if } \gamma\in D'_1;\\
(m,\; m-1),& \textrm{ if } \gamma\in D'_2;\\
(2\gamma_1+m-2,\;m-\gamma_1),& \textrm{ if } \gamma\in D_1;\\
(m-1,\;m), & \textrm{ if } \gamma\in D_2.\\
\end{cases}
$$
Thus
\small$$\sum_{\gamma\in D'_1} q^{\area(\gamma)}t^{h^+_m(\gamma)}
=q^0t^{3m-2}+q^1t^{3m-4}+q^2t^{3m-6}+\cdots+q^{m-1}t^m=\sum_{\gamma\in D_1} 
t^{\area(\gamma)}q^{h^+_m(\gamma)},
$$
$$\sum_{\gamma\in D'_2} q^{\area(\gamma)}t^{h^+_m(\gamma)}
=q^mt^{m-1}=\sum_{\gamma\in D_2} 
t^{\area(\gamma)}q^{h^+_m(\gamma)},
$$\normalsize
therefore 
$
\sum_{\gamma\in W\setminus A_0} q^{\area(\gamma)}t^{h^+_m(\gamma)}
=\sum_{\gamma\in W\setminus B_0} t^{\area(\gamma)}q^{h^+_m(\gamma)}.
$
By a similar argument as in Proposition \ref{prop:4m-1,4}, this leads to a
proof of the joint symmetry of the rational $q,t$-Catalan polynomial 
$C_{3m-1,3,1}(q,t)$.

We can express $C_{4m-1,4,1}$ in the form similar to Garsia-Haiman formula as
\tiny
$$\sigma\Big(q^{6 m}\frac{q^3}{(q - t)(q^2 - t)(q^3 - t)} - 
 q^{3 m} t^m\frac{q(t + qt + q^2 + q^2t)}{(q - t)(q^2 - t^2)(q^3 - t)}\Big)+
 q^{2 m} t^{2 m}\frac{qt (qt - 1)}{(q - t)^2(q^2 - t)(t^2 - q)}.
 $$
\normalsize
Moreover, we have the following analogue of Theorem \ref{thm:explicit-coeffs} and the proof is omitted.
\begin{thm}
For $j,k\in\N$, let $c(j,k)$ be the coefficient of $q^j t^k$ in 
$C_{4m-1,4,1}(q,t)$. For $a\in\R$, define $\lfloor a\rfloor^+=\max(\lfloor a\rfloor,0)$. 

\noindent {\rm(a)} If $4m-1\le j+k\le 6m-3$, then  
\small
$$
c(j,k)=\min \left(\Big\lfloor\frac{-6m+5+3j+k}{2}\Big\rfloor^+,\;
                 \Big\lfloor\frac{6m-1-j-k}{2}\Big\rfloor^+,\;
                 \Big\lfloor\frac{-6m+5+j+3k}{2}\Big\rfloor^+\right).
$$
\normalsize
{\rm(b)} If $j+k=4m-2$, then 
\small$$ 
c(j,k)=\min \left(\Big\lfloor\frac{-m+2+j}{2}\Big\rfloor^+,\;
                 \Big\lfloor\frac{-m+2+k}{2}\Big\rfloor^+\right).
$$
\normalsize
{\rm(c)} Otherwise $c(j,k)=0$.

\noindent As a consequence, the sequence $(c(d,0),c(d-1,1),c(d-2,2),\ldots,c(1,d-1),c(0,d))$
is unimodal for every positive integer $d$.
\end{thm}

\subsection{Gorsky and Mazin's approach}
\label{subsec:gorsky-mazin}

\def\coarea{{\rm coarea}}
\def\cowt{{\rm cowt}}
\def\wt{{\rm wt}}
\def\arm{{\rm arm}}
\def\leg{{\rm leg}}

The reader may find it helpful to compare our method with the 
following combinatorial formulation of Gorsky and Mazin's 
approach in \cite{GM2}. Let $r>3$ be an integer with $\gcd(r,3)=1$, and set $k=\lfloor r/3\rfloor$.
Define 
$$
\aligned
&X=\{(c,d)\in\N^2: 0\leq c\leq d, c\leq k, d\leq \lfloor 2r/3\rfloor\},\\
&Y=\{(a,b)\in\N^2: a+3b\leq m-1\}.\\
\endaligned
$$ 
One sees that
$X$ is the disjoint union $X_1\cup X_2\cup X_3$, 
and $Y$ is the disjoint union $Y_1\cup Y_2\cup Y_3$, where
 $$\aligned
 &X_1=\{(c,d)\in X: d\leq k\};\\
 &X_2=\{(c,d)\in X: d>k\mbox{ and }d-c\leq k\};\\
 &X_3=\{(c,d)\in X: d-c>k\};\\
 &Y_1=\{(a,b)\in Y: a+b\leq k\};\\
 &Y_2=\{(a,b)\in Y: a+b>k\mbox{ and }a+b+k\mbox{ is even}\};\\
 &Y_3=\{(a,b)\in Y: a+b>k\mbox{ and }a+b+k\mbox{ is odd}\}.\\
 \endaligned
 $$
For $(c,d)\in X$,
define 
  $\area(c,d)=m-1-(c+d)$ and
  $h^+(c,d)=$ the number of cells $x$ in the partition diagram with
    $3a(x)-ml(x)\in\{-2,-1,0,\ldots,m\}$.
For $(a,b)\in Y$, define  $\wt_1(a,b)=m-1-(a+2b)$ and
  $\wt_2(a,b)=a+b$.
Define $f:X\rightarrow Y$ by
$$
f(c,d)=
\begin{cases}
(d-c,c), & \textrm{ if } (c,d)\in X_1;\\
(3d-2k-c,c-d+k), & \textrm{ if } (c,d)\in X_2;\\
(3c-d+2k+2,d-c-k-1),& \textrm{ if } (c,d)\in X_3.\\
\end{cases}
$$
Then $f$ is a bijection that sends $\area$ to $\wt_1$ and $h^+$ to $\wt_2$.
Define $I:Y\rightarrow Y$ via $I(a,b)=(m-1-a-3b,b)$. Then $I$ is an involution 
that interchanges $\wt_1$ and $\wt_2$. Finally, 
let $g=f^{-1}\circ I\circ f$ be an involution on $X$. 
Then $g$ interchanges $\area$ and $h^+$. This implies joint symmetry of $C_{r,3,1}(q,t)$.
\end{rem}

\section{The Joint Symmetry of $C^{(m)}_5(q,t)$}
\label{sec:n=5}

For $n=5$, we define a conjectural chain map $f$ 
with domain $W\setminus T$ in Figure~\ref{fig:f5}.  
A routine but lengthy case-by-case study shows that $f$ decreases $\area$ 
by 1 and increases $\dinv_m$ by 1, and $f$ is one-to-one.
Our main obstacle to proving Conjecture~\ref{conj:reduced} for $n=5$
is that we do not know how to describe the set
$I$ of initial objects of the $f$-chains, which is 
the complement of the image of $f$ in $W$.
We leave it as an open problem to characterize $I$ and the image of $f$
explicitly, and to prove $C_T(q,t)=C_I(t,q)$.
We wrote Macaulay 2 code verifying that $C_T(q,t)=C_I(t,q)$ for $m\leq 10$.
So, the chain conjecture~\ref{conj:reduced} and joint symmetry holds
for these $m$ when $n=5$.

{\tiny
\begin{figure}[h]
%% new version from Kyungyong's 9/6 email
$$
\aligned
& \text{ if }\gamma_2\leq m\text{ then } f(\gamma)= (\gamma_0,\gamma_2,\gamma_3,\gamma_4,\gamma_1-1);\\
& \text{ if }\gamma_2>m\text{ then }\\
& \quad \left| 
\aligned
& \text{ if }\gamma_4\geq m\text{ then }\\ 
&\quad \left| 
\aligned
& \text{ if }\gamma_3-\gamma_1\leq m\text{ then } f(\gamma)= (\gamma_0,\gamma_1,\gamma_3,\gamma_4,\gamma_2-1); \\
& \text{ if }\gamma_3-\gamma_1> m\text{ then }\\
& \quad \left|  
\aligned 
&  \text{ if }\gamma_4-\gamma_2\leq m\text{ then }\\
%& \quad \left|  
%\aligned 
% &  \text{ if } m+1<\gamma_3-\gamma_1\leq 2m\text{ then } \\
%& \quad \left|  
%\aligned 
%&  \text{ if }\gamma_2-\gamma_1> m\text{ then }
% f(\gamma)= (\gamma_0,\gamma_1,\gamma_4+1,\gamma_2-1,\gamma_3-1);\\
% &  \text{ if }\gamma_2-\gamma_1\leq m\text{ then }\\
 & \quad \left|  
\aligned 
 &  \text{ if }\gamma_3-\gamma_4\leq m+1\text{ then }f(\gamma)= (\gamma_0,\gamma_1,\gamma_2,\gamma_4,\gamma_3-1); \\
 &  \text{ if }\gamma_3-\gamma_4>  m+1\text{ then } f(\gamma)= (\gamma_0,\gamma_1,\gamma_4+1,\gamma_2-1,\gamma_3-1);\\
%\endaligned\right.\\%\endaligned\right.\\
%& \text{ if } \gamma_3-\gamma_1<m+2 \text{ or }\gamma_3-\gamma_1> 2m\text{ then }\\
% & \quad \left|  
%\aligned 
% &  \text{ if }\gamma_3-\gamma_4\leq m+1\text{ then }f(\gamma)= (\gamma_0,\gamma_1,\gamma_2,\gamma_4,\gamma_3-1); \\
% &  \text{ if }\gamma_3-\gamma_4>  m+1\text{ then } f(\gamma)= (\gamma_0,\gamma_1,\gamma_4+1,\gamma_2-1,\gamma_3-1);\\
%\endaligned\right. \\
\endaligned\right.\\
&  \text{ if }\gamma_4-\gamma_2> m\text{ then }  f(\gamma)= (\gamma_0,\gamma_1,\gamma_2,\gamma_3,\gamma_4-1); \\
\endaligned\right.\\
\endaligned
\right.\\
& \text{ if }\gamma_4<m\text{ then }\\
&\quad \left| 
\aligned
& \text{ if }\gamma_3-\gamma_1\leq m\text{ then }\\
& \quad \left|  
\aligned
&\text{ if }\gamma_2-\gamma_4\leq m+1\text{ then } f(\gamma)= (\gamma_0,\gamma_1,\gamma_3,\gamma_4,\gamma_2-1); \\
&\text{ if }\gamma_2-\gamma_4> m+1\text{ then } \\
& \quad \left|  
\aligned 
&\text{ if }\gamma_3-\gamma_4\leq m+1\text{ then } \\
& \quad \left| \aligned &\text{ if }\gamma_2-\gamma_3>m\text{ or }\gamma_3\le m\text{ then } f(\gamma)= (\gamma_0,\gamma_4+1,\gamma_3,\gamma_1-1,\gamma_2-1); \\
&\text{ if }\gamma_2-\gamma_3\le m\text{ and }\gamma_3> m\text{ then } f(\gamma)= (\gamma_0,\gamma_4+1,\gamma_1,\gamma_3-1,\gamma_2-1); \\
\endaligned
\right.\\
&\text{ if }\gamma_3-\gamma_4> m+1\text{ then }
%\\
%& \quad \left|  
%\aligned 
%&\text{ if } \gamma_2-\gamma_1\leq m-1, \quad \gamma_3-\gamma_1\leq m-1,  \quad \gamma_2-\gamma_4\geq m+3, \\
%& \quad \quad\text{ and }\gamma_3-\gamma_4\geq m+3  \text{ then }   
f(\gamma)= (\gamma_0,\gamma_4+1,\gamma_1,\gamma_2-1,\gamma_3-1); \\
%&\text{ else }
% f(\gamma)= (\gamma_0,\gamma_4+2,\gamma_1-1,\gamma_2-1,\gamma_3-1) \text{ or }; \\
%\endaligned
%\right.
\endaligned
\right.
\endaligned
\right.\\
& \text{ if }\gamma_3-\gamma_1> m\text{ then }\\
& \quad \left|  
\aligned
&\text{ if }\gamma_2-\gamma_4\leq m+1\text{ then }\\
& \quad \left| 
\aligned 
&  \text{ if }\gamma_3-\gamma_4\leq m+1\text{ then }%\\
%& \quad \left|  
%\aligned 
%&  \text{ if }\gamma_4-\gamma_2\leq m\text{ then }  
f(\gamma)= (\gamma_0,\gamma_1,\gamma_2,\gamma_4,\gamma_3-1);\\
% \text{ or }f(\gamma)= (\gamma_0,\gamma_1,\gamma_2,\gamma_4,\gamma_3-1);
%\\
%&  \text{ if }\gamma_4-\gamma_2> m\text{ then }  f(\gamma)= (\gamma_0,\gamma_1,\gamma_2,\gamma_3,\gamma_4-1); \\
%\endaligned
%\right.\\
& \text{ if }\gamma_3-\gamma_4> m+1\text{ then } f(\gamma)= (\gamma_0,\gamma_4+1,\gamma_1,\gamma_2-1,\gamma_3-1)\\
\endaligned
\right.\\
&\text{ if }\gamma_2-\gamma_4> m+1\text{ then }\\
& \quad \left|  
\aligned 
%&\text{ if }\gamma_3-\gamma_4\leq m+1\text{ then } f(\gamma)= (\gamma_0,\gamma_4+1,\gamma_3,\gamma_1-1,\gamma_2-1); \\
&\text{ if }\gamma_3-\gamma_4> m+2\text{ and } \gamma_2-\gamma_4> m+2\text{ then } \\
&\hspace{100pt} f(\gamma)= (\gamma_0,\gamma_4+2,\gamma_1-1,\gamma_2-1,\gamma_3-1); \\
&\text{ if }\gamma_3-\gamma_4\leq m+2\text{ or } \gamma_2-\gamma_4= m+2\text{ then } \\ &\hspace{100pt} f(\gamma)= (\gamma_0,\gamma_1,\gamma_4+1,\gamma_2-1,\gamma_3-1). \\
\endaligned
\right.
\endaligned
\right.
\endaligned
\right.\endaligned
\right.\endaligned
$$ 
\caption{The conjectured chain map for $n=5$.}
\label{fig:f5}
\end{figure}
}

\begin{ack}
The authors are grateful to Eugene Gorsky, Mikhail Mazin and Luis Sordo Vieira 
for valuable discussions. Gorsky's comments prompted us to prove
Conjecture~\ref{conj:qtcat} for $n\leq 4$ by using the results
of~\S\ref{sec:proof-jsymm}. The authors are also grateful to the
anonymous referees for many useful comments.
\end{ack}

\newpage


\end{document}